\documentclass[a4paper,12pt]{amsart}
\usepackage[T1]{fontenc}
\usepackage{lmodern}
\usepackage{amsthm}
\usepackage{amssymb}
\usepackage{amsmath,amscd}
\usepackage{eufrak}
\newtheorem{definition}{Definition}

\newtheorem{lem}[definition]{Lemma}
\newtheorem{thm}[definition]{Theorem}

\newtheorem{prop}[definition]{Proposition}
\newcommand{\C}{\mathbb{C}}

\newcommand{\Z}{\mathbb{Z}}

\newcommand{\R}{\mathbb{R}}

 \newtheorem{lemma}[definition]{Lemma}
 \newtheorem*{cor}{Corollary}

\theoremstyle{definition}  
\theoremstyle{definition}  
\theoremstyle{remark}  
\theoremstyle{remark}  
\theoremstyle{remark}  

\theoremstyle{plain} \newtheorem*{thm*}{Theorem}
\theoremstyle{plain} \newtheorem*{cor*}{Corollary}

\newcommand{\LLL}{\tilde{\mathcal{L}}}
\renewcommand {\epsilon}{\varepsilon}
\newcommand {\F} {\mathcal{F}}
\renewcommand {\d}[2] {\frac{\partial #1}{\partial #2}}

\title[Pfaffian definitions of elliptic functions]{Pfaffian definitions of Weierstrass elliptic functions}
\author{Gareth Jones}
\address{School of Mathematics, University of Manchester, Oxford Road, Manchester, M13 9PL, UK.} \address{gareth.jones-3@manchester.ac.uk}
\author{Harry Schmidt}
\address{harry.Schmidt@manchester.ac.uk}
\begin{document}
\maketitle
%\tableofcontents
\begin{abstract}
We give explicit definitions of the Weierstrass elliptic functions $\wp$ and $\zeta$ in terms of pfaffian functions, with complexity independent of the lattice involved. We also give such a definition for a modification of the Weierstrass function $\sigma$. We give some applications, and in particular, answer a question of Corvaja, Masser, and Zannier on additive extensions of elliptic curves. \\

%Nous donnons des d\'efinitions explicites des fonctions elliptiques de Weierstrass $ \wp$ et $ \zeta $ en termes de fonctions pfaffiennes, avec une complexit\'e ind\'ependante du r\'eseau impliqu\'e. Nous donnons aussi une telle d\'efinition pour une modification de la fonction Weierstrass $ \sigma $. Nous d\'ecrivons quelques applications, et en particulier, r\'epondons \'a une question de Corvaja, Masser et Zannier sur les extensions additives de courbes elliptiques
\end{abstract}
\section{Introduction}

%{\bf NUMBERING NEEDS SORTING OUT}

% go over the appendix and make the propositions more precise
Khovanskii's theory of pfaffian functions provides zero estimates for real analytic functions satisfying certain differential equations. Gabrielov and Vorobjov then extended the theory to include, among other things, effective stratification results for varieties defined by pfaffian functions. Following the presentation in \cite{GV}, we say that a sequence $f_1,\ldots,f_l:U\to \R$ of analytic functions on an open set $U$ in $\R^n$ is a \emph{pfaffian chain} if, for $i=1,\ldots,l$ and $j=1,\ldots,n$ there are real polynomials $p_{i,j}$ in $n+i$ indeterminates such that
\[
\d{f_i}{x_j} (x) = p_{i,j}(x,f_1(x),\ldots,f_i(x))
\]
on $U$. We say that a function $f$ is \emph{pfaffian} if there is a polynomial $p$ such that $f$ is $p(x,f_1(x),\ldots,f_l(x))$. Pfaffian functions come equipped with a notion of complexity. We say that $f$ as above has \emph{order} $l$ and degree $(\alpha,\beta)$, where $\alpha$ is a bound on the maximum of the degrees of the $p_{i,j}$ and $\beta$ is a bound on the degree of $p$. Khovanskii proved bounds on the number of connected components of $f$, in terms of the complexity of $f$, provided that the domain $U$ is a sufficiently simple set, such as a product of open intervals.

Macintyre \cite{Macintyre} was the first to observe a connection between pfaffian functions and elliptic functions. Extending Macintyre's work, we will give explicit definitions (in the sense of first-order logic) of Weierstrass elliptic functions in terms of pfaffian functions. In these definitions we identify $\C$ with $\R^2$.  We express our results without the language of logic. But this does necessitate further definitions. First, we say that $U\subseteq \R^n$ is a \emph{simple domain} if $U$ is the image of a product of open intervals under an invertible affine transformation. Now let $X\subseteq \R^n$. Suppose that $U_i\subseteq \R^n$ are simple domains, for $i=1,\ldots, L$, and that for each $i$ we have pfaffian functions $f_{i,1},\ldots,f_{i,m_i}:U_i\to \R$ with a common chain of order $r$ and degree $(\alpha,\beta)$.  And suppose that $m_i\le M$  and that
\[
X=\bigcup_{i=1}^L \{ x \in U_i : f_{i,1}(x)=\cdots= f_{i,m_i}(x)=0\}
\]
Then we call $X$ \emph{piecewise semipfaffian} (for want of a shorter phrase that hasn't been taken), of \emph{format} $(r,\alpha,\beta, n , L,M)$. We will also need projections of these sets. So if $X$ is as above, and $Y=\pi X $ where $\pi : \R^n \to \R^m$ is the projection onto the first $m$ coordinates, then we call $Y$ a \emph{piecewise subpfaffian} set of \emph{format} $(r,\alpha,\beta, n , L,M)$ (so the measure of complexity of $Y$ is that of $X$).

Recall that given a lattice $\Omega \subseteq \C$ there is an associated Weierstrass $\wp$-function given by
\[
\wp_{\Omega} (z) =\frac{1}{z^2} +\sum \left(\frac{1}{(z-\omega)^2}-\frac{1}{\omega^2}\right)
\]
where the sum is taken over nonzero $\omega$ in the lattice. We fix a pair of elements $\omega_1, \omega_2 \in \Omega$ that form a basis of $\Omega $ such that $\tau = \omega_2/\omega_1 $ lies in the upper half plane and satisfies $|\Re(\tau)| \leq \frac12$ and $|\tau| \geq 1$. We associate to $\Omega$ the set
\begin{align*}
\mathfrak{F}_\Omega=\{r_1\omega_1+r_2\omega_2; r_1, r_2\in [0,1), r_1^2+r_2^2\neq 0\},
\end{align*}
which is a fundamental domain for $\Omega$ with 0 removed. We prove the following.
\begin{thm}\label{defp} On  $\frak{F}_\Omega$, the graph of $\wp|_{\frak{F}_\Omega}$ is a piecewise semipfaffian set of format $(7,9,1,4,144503,2)$.
\end{thm}

This result opens up the possibility of applying results on pfaffian functions (for a survey, see \cite{GV}) to elliptic functions in an effective  manner. So, for instance, the second author has recently found \cite{Galois} a new proof of polynomial Galois bounds for torsion on elliptic curves; a result previously established  by Masser \cite{Masser} and David \cite{David}. This proof combines our work here with ideas of Pila \cite{Pila1}, \cite{Pila2} and the first author and Thomas \cite{JTsurface} on counting problems for pfaffian functions. Our work also potentially extends the applicability of the methods used by Binyamini and Novikov in their recent breakthrough  on Wilkie's conjecture \cite{BN}. Their methods allow for a certain uniformity and in combination with the results of this paper this might lead to a counting result for sets definable (by suitably simple formulas) in an expansion of the real ordered field by restrictions of Weierstrass elliptic functions  that is effective and uniform in the lattice. We mention some other possible applications below. For now, we record the following  immediate consequence of our first theorem, together with Khovankii's theorem (in the form of \cite[Corollary 3.3]{GV}); the following explicit uniform zero estimate.
\begin{cor} Suppose that $P$ is a polynomial in two variables, with complex coefficients, not identically zero and of total degree bounded by $T\ge 20$. Then, on $\frak{F}_\Omega$, the function $P(z,\wp(z))$ has at most
$$
7.5373\times 10^{14}T^{11}
$$
zeroes.
\end{cor}

This estimate is presumably far from optimal in its dependence on the degree, not to mention the constant. If the constant is allowed to depend on the lattice, then the 11 can be replaced with a 2. See for instance Exercise 22.21 in Masser's recent book \cite{MasserBook} for such an estimate. And see also the paper \cite{BM} by Brownawell and Masser for further discussion on optimality of zero estimates (also for the exponential function). \\

Another important function associated to a lattice is the Weierstrass zeta function $\zeta_\Omega$ which is the unique odd (meromorphic) function whose derivative is equal to $ -\wp_{\Omega}$.  Consequently it is not periodic but quasi-periodic with respect to the lattice. That is
\begin{align*}
\zeta_{\Omega}(z + \omega) = \zeta_{\Omega}(z) + \eta(\omega), ~~\omega \in \Omega
\end{align*}
where $\eta$ is a group homomorphism from $\Omega$ to $\C$.
Even so restricting $\zeta_{\Omega}$ to a fundamental domain for the action of $\Omega$ is sufficient to describe the graph of $\zeta_{\Omega}$.  We prove the following for this restriction.
\begin{thm}\label{defz}
On $\frak{F}_\Omega$, the graph of $\zeta|_{\frak{F}_\Omega}$ is a piecewise subpfaffian set of format $(9,9,1,6,144503,4)$.
\end{thm}

As with $\wp$ above, this leads immediately to a uniform explicit zero estimate for $\zeta$ on $\frak{F}_\Omega$, again presumably far from optimal (compare with Exercises 22.22 and 22.23 in \cite{MasserBook}). We omit this, and instead give the application which began our interest in the problems considered here. This is motivated by a result of Corvaja, Masser and Zannier \cite[Theorem 1]{CMZ}. Suppose that $E$ is an elliptic curve over the complex numbers and that $G$ is an extension of $E$ by $\mathbb{G}_a$ with no nonconstant regular functions. Let $\frak{K}$ be the maximal compact subgroup of $G$. Corvaja, Masser and Zannier proved that if $\mathcal{C}$ is a curve in $G$ then the intersection $\mathcal{C}\cap \frak{K}$ is finite. In fact they showed that the size of the intersection is bounded by a function of the degree of the curve $\mathcal{C}$. They asked about what shape this function $c(d)$ could take. Using Khovanskii's theorem together with our results above, we are able to show that we can take $c(d)$ to be
$$
1.333\times 10^{25}d^{19}
$$
for $d\geq 3$.
So the bound is completely independent of the curve $E$. As Corvaja, Masser and Zannier noted, their result can be interpreted as a sharpening of Manin-Mumford for additive extensions of elliptic curves, and the above gives an explicit  bound for Manin-Mumford that is independent of $G$. To the best of our knowledge, this is the first instance of such a uniform and explicit Manin-Mumford result. For more details, and an extension to extensions of products of elliptic curves, see \cite{JSprep}.\\

Finally the function  highest in the hierarchy of Weierstrass function is the sigma function $\sigma_{\Omega}$. It vanishes at all points of $\Omega$ and its logarithmic derivative equals $\zeta$. We will show in the appendix that $\sigma$ does not have a definition whose complexity is uniform in the lattice. Instead we will investigate the function $\varphi_{\Omega}$ defined by
\begin{align*}
\varphi_{\Omega}=\exp(-\frac12z^2\eta_1/\omega_1 + \pi i z/\omega_1)\sigma_{\Omega},
\end{align*}
where $\eta_1 = \eta(\omega_1)$. This function is periodic in $\omega_1$ and transforms as follows with translations by $\omega_2$
\begin{align*}
\varphi_{\Omega}(z + \omega_2) = -\exp(-2\pi iz)\varphi_{\Omega}
\end{align*}
\cite[p.246, Theorem 3']{Lang} and by iteration
\begin{align}
\varphi_{\Omega}(z + n\omega_2) =(-1)^n\exp(-2\pi i nz/\omega_1 -\pi i n(n-1)\omega_2/\omega_1)\varphi_{\Omega}(z)
\end{align}
for an integer $n\geq 0$ and then by translation
\begin{align}
\varphi_{\Omega}(z - n\omega_2) =(-1)^n\exp(2\pi i nz/\omega_1 -\pi i n(n+1)\omega_2/\omega_1)\varphi_{\Omega}(z)
\end{align}
Hence describing the function on $\frak{F}_\Omega$ leads to an  understanding of its graph and we prove the following.
\begin{thm}\label{defs} The graph of $\varphi$ restricted to $\frak{F}_\Omega$ is a piecewise subpfaffian set of format  $(17,9,6,10,114565235503,8 )$.
\end{thm}

This $\varphi$ is again connected with extensions of elliptic curves, this time by the multiplicative group. And in future work, we will apply our theorems, together with recent work by Margaret Thomas and the first named author \cite{JTcounting} to prove certain effective instances of relative Manin-Mumford for semi-constant families of multiplicative extensions of a fixed elliptic curve as in \cite{BMPZ}.  And again, there is potential for uniformity here, at least in the case that the elliptic curve has complex multiplication.

Let $\wp=\wp_\Omega$ be as above, and let $e_1,e_2$ and $e_3$ be the zeros of the associated polynomial $4w^3-g_2w-g_3$ where $g_2(\Omega)= 60 \sum \omega^{-4}$ and $g_3(\Omega)= 140\sum \omega^{-6}$ with the summations over nonzero periods. Macintyre proved that on a simply connected open set $U$ in $\C\setminus \{ e_1,e_2,e_3\}$ the real and imaginary parts of any fixed branch of the inverse of $\wp$ are pfaffian. We would like to use this as follows. Start with a small disc in $\C\setminus\{ e_1,e_2,e_3\}$ and fix a branch of the inverse of $\wp$ on the disc that takes values in a fixed fundamental domain $F$. Now continue this branch of the inverse to a half-plane bounded by a horizontal or vertical line through some $e_i$, or if there isn't such a continuation then instead continue to a strip bounded by two of these lines. Then repeat this process until we've covered $\C\setminus \{ e_1,e_2,e_3\}$ with half-planes and strips on which we have branches of the inverse. Clearly the number of domains needed is uniform in $\Omega$. And by Macintyre's result, all the functions involved have pfaffian real and imaginary parts. We would like to then define $\wp|_F$ by translating the values of the inverse back into $F$ whenever they happen to fall outside $F$. But this introduces a potential lack of both effectivity and uniformity, for we don't know in advance how many translations we will need (or even if the number of translations needed is finite). This is the problem we solve, at least in a sufficiently general special case. We do not work directly with the $\wp$ functions as above, rather we work with the functions associated to the Legendre curves $E_\lambda$ defined by $Y^2=X(X-1)(X-\lambda)$, for complex numbers $\lambda\neq 0,1$. Let $\Gamma=\{ \lambda \in \C\setminus \{0,1\}; |\lambda| \leq1,|1-\lambda|\leq1\}$. Then for $\lambda$ in the interior of $\Gamma$ we can define analytic functions
\[
\pi F(\lambda), ~~i \pi F(1-\lambda)
\]
where $F(\lambda) = F(1/2,1/2,1;\lambda)$ is a classical hypergeometric function
\begin{align*}
F(\lambda) = \sum_{n= 0}^{\infty}\frac{(\frac 12)_n^2}{n!^2}\lambda^n
\end{align*}
($(\frac 12)_n = \frac12(\frac12 +1)\dots (\frac12 +n-1)$).  It is known that $\omega_1 =\pi F(\lambda)$ and $\omega_2=i\pi F(1-\lambda)$ form a basis of the lattice associated to the differential $\frac{dX}{2Y}$. % and we recall a proof of this in section \ref{analyticcontinuation} as a warm up for later arguments.
 For our inverse function we start with
\begin{align}
z(\lambda,\xi)= \int_\xi^{-\infty} \frac{dX}{2\sqrt{X(X-1)(X-\lambda)}}\label{inverse}
\end{align}
where $\xi$ lies in the open interval $(-\infty,0)$. With these definition we can introduce the Betti coordinates (following Bertrand's terminology) as follows:
\begin{align}\label{bettidefinition1}
b_1 & = \frac{\bar{\omega_2}z-\omega_2\bar{z}}{A}\\
b_2 & = \frac{\omega_1\bar{z}-\bar{\omega_1}z}{A}\label{bettidefinition2}
\end{align}
where we set
\[
A=\omega_1\bar{\omega_2}-\omega_2\bar{\omega_1}.
\]
These Betti coordinates increase as we pass through further fundamental domains, and our main technical result for $\wp$ is an explicit bound on $|b_1|$ and $|b_2|$, with $\xi$ as above and $\lambda \in \Gamma$ also such that $\Re \lambda \le 1/2$. This can then be extended to bounds that hold for all $\xi$, via a topological argument on suitably continuing $z$. To bound $|b_1|$ and $|b_2|$ we first produce a lower bound for $|A|$, establishing an explicit lower bound of the form $|A|\gg \log |\lambda|^{-1}$ for small $\lambda \in \Gamma$. This may be of some independent interest, also because of the connection of $A$ to the Faltings height of $E_\lambda$ (for algebraic $\lambda$). To bound the Betti coordinates we then bound the numerators in \eqref{bettidefinition1}, \eqref{bettidefinition2}. Standard estimates would give upper bounds of order $\left(\log |\lambda|^{-1}|\right)^2$, where we stay with small $\lambda$ in $\Gamma$. In order to bound $|b_1|,|b_2|$ by an absolute constant, we show that in fact some cancellation occurs in the numerators and we can remove a power of the logarithm. The argument here is rather technical and we delay further discussion until later.

%Our bound on Betti coordinates leads to an effective and uniform version of the Rouch\'e bound used in \cite{Masser-Zannier} and \cite{Barroero-Capuano}. This may be of interest in connection with effectivity in relative Manin-Mumford for this circle of problems, via Binyamini's recent work \cite{}.

Once we have established bounds on $|b_1|$ and $|b_2|$ we can proceed more or less as in the sketch above to give a definition for the function $\wp_\lambda$ associated to $E_\lambda$, at least for those $\lambda$ for which our bounds hold. But we can then extend easily to the general case, as every elliptic curve over $\C$ is isomorphic to $E_\lambda$ for such a $\lambda$. We can then give definitions for $\zeta$, using elliptic integrals of the second kind (again following Macintyre). \\

 Betti maps of this type were introduced in a paper by Masser and Zannier \cite{MZ}, and the terminology comes from a paper by Bertand, Masser, Pillay and Zannier \cite{BMPZ}. In these papers the maps are a tool in the study of unlikely intersection problems. Very recently, the maps themselves have been studied, for instance in a paper by Corvaja, Masser and Zannier \cite{CMZ2}, and a paper by Voisin \cite{Voisin}, and also ongoing work by Andr\'e, Corvaja and Zannier.  These maps were also implicitly used in older work, in particular in Manin's famous proof of the Mordell conjecture over function fields \cite{Manin}. \\

We now return to the Weierstrass functions. We cannot directly handle $\varphi = \varphi_{\Omega}$ but have to pass to its logarithm whose derivative is given by an expression involving $\zeta$. Then we can use the definition of $\zeta$ to locally define the logarithm and  compose with the exponential function to define $\varphi$ locally. The main new technical problem here is that as we continue the logarithm we might pass through many fundamental domains of the exponential function; again a potential threat to both uniformity and effectivity. We give an explicit absolute bound on the imaginary part of this continuation for $\varphi_\lambda$ associated to $E_\lambda$. \\

%\textit{Our interest in this problem came from the desire to find an effective description for the intersection of an algebraic variety with sets implicitly defined by elliptic functions. For example as a direct consequence of \cite[Theorem 1]{CMZ} and our work we get an explicit bound for the intersection of a curve with the maximal compact subgroup of a universal vectorial extension of an elliptic curve that only depends on the degree of the curve and not the elliptic curve. Due to the quality of Khovanskii's bounds it is furtermore polynomial in the degree of the curve. \textbf{state bound?}}
%
%\textit{Another consequence is that the number of points in the intersection of a curve with the graph of $\wp$ restricted to $\mathfrak{F}_\Omega$ is bounded by a fixed polynomial in the degree of the curve.\textbf{Maybe omit that?}}\\
%\\

In addition to the potential for diophantine applications, there is some possibility of using our work to study elliptic functions from the logical point of view. Macintyre has (see \cite{MacintyreDecide}) established decidability for the theory of the expansion of the real field by a restricted $\wp$-function, assuming suitable transcendence conjectures. Pfaffian ideas play an important role in Macintyre's proof and it could be interesting to revisit the proofs in search of uniformity.

This is how the article is organized. In the next section we investigate the functions we are working with more closely and state the propositions.  Then in section 3 we prove several crucial lower bounds. In section 4 we reduce Theorem \ref{defp},\ref{defz} and \ref{defs} to  Propositions \ref{propbetti}, \ref{propsigma} stated in section 2. In section 5 we prove the necessary upper bounds for the conclusion of Proposition \ref{propbetti} in section 6. Then in section 7 and 8 we prove the upper bounds necessary to conclude Proposition \ref{propsigma} in section 9. Finally, in the appendix we point out some limits of how far our results can be pushed. And then we conclude with some further remarks on the Betti map. \\

We are grateful to Angus Macintyre for several fruitful conversations during the early stages of this work. And we also thank Pietro Corvaja, David Masser and Umberto Zannier for helpful comments on a draft version of this paper and for the suggestion to include further remarks on the Betti map. We also thank Gabriel Dill for pointing out several typos in a previous draft of this article. The second-named author would like to thank the Mathematical Institute of the University of Oxford where the main bulk of his contribution to this work was done while he was staying there as a fellow of the Swiss National Science Foundation. Both authors thank the Engineering and Physical Sciences Research Council for support under grant  EP/N007956/1.

\section{Analytic continuation}\label{analyticcontinuation}
The definition of the inverse in (\ref{inverse}) is ambiguous and it is important for us to have fixed definitions of the functions involved in the Betti-coordinates for our explicit estimates. So we introduce the set $X_\lambda = \C \setminus ((-\infty,0]\cup L_\lambda\cup [1,\infty))$  where $L_\lambda$ is a straight closed line joining 0 and $\lambda$ and we will define the inverse  on this simply connected set. \\

We will also need definitions of $\omega_1$ and $\omega_2$ as closed elliptic integrals for the estimates in section \ref{numerator}, and we begin by giving these. They are equal to
\begin{align}\label{ellintegral}
\omega_1 = \int_{1}^{\infty}\frac{dX}{\sqrt{X(X-1)(X-\lambda)}}, ~ \omega_2 = \int_{0}^{-\infty}\frac{dX}{\sqrt{X(X-1)(X-\lambda)}}
\end{align}
where, for both, we take the integral over the real line.
These equalities are proven in \cite{Hus}. %but we note that one has to be careful about the choice of the path and that it is not a trivial matter that we indeed can take the real line as a path of integration.
However, we also want to make the choice of a square-root  clear.  For $\omega_1$ we chose the standard square-root on the complex plane sliced along the negative real axis such that $\sqrt{1} = 1$  and for $\omega_2$ we chose the  square-root on the complex plane sliced along the positive real axis such that $\sqrt{-1} = i$. By restricting to $\lambda \in (0,1)$ we see that we have made the right choice.  For our investigations we also need the following two equalities
\begin{align} \label{ellintegral2}
-\omega_1 = \int_0^{\lambda} \frac{dX}{\sqrt{X(X-1)(X-\lambda)}}, ~~ \omega_2 = \int_{\lambda}^1 \frac{dX}{\sqrt{X(X-1)(X-\lambda)}}.
\end{align}
 Here we pick the same square-root for $\omega_1,\omega_2$ as in (\ref{ellintegral}). \\

We write $\Omega_\lambda$ for the lattice spanned by $\omega_1,\omega_2$. It is well-known that the corresponding Weierstrass-invariants are
\begin{align*}
 g_2 = \frac43(\lambda^2 - \lambda +1), ~~ g_3 = \frac4{27}(\lambda -2)(\lambda +1)(2\lambda -1).
 \end{align*}

 We denote by $\wp_\lambda, \zeta_\lambda,$ and $ \varphi_\lambda$ the Weierstrass functions associated to $\Omega_\lambda$. Also well-known is the fact that (\ref{inverse}) is a local inverse for $\wp_\lambda +\frac13(\lambda+1)$. We choose the square root for (\ref{inverse}) to be such that $\sqrt{-1} = i$ (the same as for $\omega_2$) and for each fixed $\lambda$ we continue  $z(\xi) =z(\lambda,\xi)$ as an analytic function of $\xi$ to $X_\lambda$ by continuing it north on the complex plane from $(-\infty,0)$. The function $z(\lambda, \xi)$ has well-defined limits $z(\lambda,1), z(\lambda,0)$ and from (\ref{ellintegral}) we deduce that
\begin{align*}
z(\lambda,0) = \omega_2/2, ~~ z(\lambda,1) = \omega_1/2.
\end{align*}
In order to prove (\ref{ellintegral2}) we note that from the above follows that $\wp_\lambda(\omega_1/2) +\frac13(\lambda+1) =1$, $\wp_\lambda(\omega_2/2) +\frac13(\lambda+1) =0$ and so consequently $\wp_\lambda((\omega_1+\omega_2)/2) +\frac13(\lambda+1) =\lambda$. Further
\begin{align}
dz = -\frac{dX}{2\sqrt{X(X-1)(X-\lambda)}}\label{dz}
\end{align}
and so $\int_0^{\lambda} \frac{dX}{\sqrt{X(X-1)(X-\lambda)}} =-\int_{\omega_2/2}^{(\omega_1+\omega_2)/2}dz =-\omega_1/2$ and similarly for the other equality in (\ref{ellintegral2}). \\

With (\ref{ellintegral}) and (\ref{inverse}) we can continue $\omega_1, \omega_2$, and $z$ for fixed $\xi \in X_\lambda$, as analytic functions of $\lambda$ to a small neighbourhood of $\Gamma$ in $\C\setminus \{0,1\}$ and call them analytic on $\Gamma$. From this point on we will work with the so established analytic function $z(\lambda, \xi)$ of two variables on the fibred product $\Gamma\times_\lambda X_\lambda$. But  for our purposes it actually suffices to work on $\mathcal{F}\times_\lambda X_\lambda$ where
\begin{align*}
\mathcal{F} = \{\lambda \in \Gamma; \Re(\lambda) \leq \frac12\}.
\end{align*}
We will deduce Theorem \ref{defp} from the following proposition.

\begin{prop} \label{propbetti} Let $\omega_1, \omega_2$  and $z$ be defined as above on $\mathcal{F}\times_\lambda X_\lambda$. Let $b_1,b_2$ be as in (\ref{bettidefinition1}), (\ref{bettidefinition2}). Then
\begin{align*}
\max\{|b_1|,|b_2|\}\leq 42.
\end{align*}
\end{prop}

The number in the estimate is unlikely to be the best possible. But it does at least contain the answer to other fundamental questions.

We show in the appendix that the bound necessarily depends on the choice of the fundamental domain.\\

We now turn to the definition of $\varphi$. With $\omega_1$ and  $\omega_2$ defined above we define $\varphi_\lambda$ by
\begin{align*}
\varphi_\lambda=\exp\left(-\frac12\eta_1\omega_1 (z/\omega_1)^2 + \pi i z/\omega_1\right)\sigma_\lambda
\end{align*}
where $\sigma_\lambda$ is the Weierstrass sigma function associated to the lattice $\Omega_\lambda$.
We also set $\eta_1 = 2\zeta_\lambda(\omega_1/2)$ and $\eta_2 = 2\zeta_\lambda(\omega_2/2)$ to be the quasi-periods associated to $\omega_1$ and $\omega_2$, respectively. These satisfy the following relations
\begin{align}
\eta_1 & = \frac13(1 -2\lambda)\omega_1 + 2\lambda(1-\lambda)\omega_1'\label{eta1}\\
\eta_2 & = \frac13(1-2\lambda)\omega_2 +2\lambda(1-\lambda)\omega_2'\label{eta2}
\end{align}
(see for example \cite{thesis}). These will play a crucial role in the investigations of $\varphi_\lambda$ and, in the appendix, of $\sigma_\lambda$. The function $\varphi_\lambda$ itself satisfies the following differential equation
\begin{align}
\varphi_\lambda'/\varphi_\lambda = -z\eta_1/\omega_1 + \pi i/\omega_1 + \zeta_\lambda.\label{diffphi}
\end{align}

We can write  $\zeta_\lambda(z) - \eta_1/2 = -\int_{\omega_1/2}^{z}\wp_\lambda(t) dt $ and as $\wp_\lambda$ has no residues this integral is independent of the choice of the path.  Setting $t = z(\lambda,X)$ and using (\ref{dz}) we obtain
\begin{align*}
\zeta_\lambda(z(\lambda, \hat{X})) = \int_{1}^{\hat{X}}\frac{(X-\frac13(\lambda+1))dX}{2\sqrt{X(X-1)(X-\lambda)}}  +\eta_1/2
\end{align*}
for $\hat{X} \in X_\lambda$. Similarly, we have
\begin{align*}
z(\lambda,\hat{X}) = -\int_1^{\hat{X}}\frac{dX}{2\sqrt{X(X-1)(X-\lambda)}} + \omega_1/2
\end{align*}
and so with (\ref{diffphi}) we deduce that
\begin{align*}
\varphi_\lambda'/\varphi_\lambda(z(\lambda,\hat{X})) & = \eta_1/\omega_1\int_1^{\hat{X}}\frac{dX}{2\sqrt{X(X-1)(X-\lambda)}} +\int_{1}^{\hat{X}}\frac{(X-\frac13(\lambda+1))dX}{2\sqrt{X(X-1)(X-\lambda)}}    +\pi i/\omega_1.
\end{align*}

We can't define $\log(\varphi_\lambda(z)) -\log(\varphi_\lambda(\omega_1/2))$ by the integral $\log(\varphi_\lambda(z)) -\log(\varphi_\lambda(\omega_1/2))=\int_{\omega_1/2}^{z}\varphi_\lambda'/\varphi_\lambda(t)dt$ independently of the path of integration as $\varphi_\lambda'/\varphi_\lambda$ has residue 1 at every point of the lattice $\Omega_\lambda$.  However these are the only points where it has non-zero residue. The image of $X_\lambda$ by $z$ does not contain any lattice points. Thus an integral of $\varphi_\lambda'/\varphi$ along a closed path in that image vanishes. This allows us to define
\begin{align}
\log(\varphi_\lambda(z(\lambda, \xi))) - \log(\varphi_\lambda(\omega_1/2))=- \int_{1}^{\xi}\frac{\varphi_\lambda'/\varphi_\lambda(z(\lambda,\hat{X}))d\hat{X}}{2\sqrt{X(X-1)(X-\lambda)}} \label{phi}
\end{align}
unambiguously for any $\xi \in X_\lambda$ where, except for the endpoint $1$, the path of integration lies entirely in $X_\lambda$. To ease notation we define  $\mathcal{L}$ to be equal to the left hand side of the equation, as defined by the right hand side.   Clearly $\exp(\mathcal{L})$ is equal to $\varphi_\lambda$ times a constant independent of $\xi$. We will find a pfaffian definition of the graph $\varphi_\lambda$  in this manner with the help of the following proposition.
\begin{prop}\label{propsigma} Let $\mathcal{L}$ be defined as above on $\mathcal{F}\times_\lambda X_\lambda$. Then
\begin{align*}
|\Im\left(\mathcal{L}\right)/2\pi| \leq  384.
\end{align*}
\end{prop}

%{\bf fix me}
%\begin{lem} For $\lambda \in \mathcal{F}$, $|\omega_1| \geq 1, |\Re(\omega_2/\omega_1)| \leq \frac12 $.
%\end{lem}

\section{Bounding the area from below}\label{numerator}
In this section we  investigate properties of the periods $\omega_1$ and $\omega_2$ and the area (up to $\pm2i$ ) $A$. We prove a lower bound for $|A|$ and show that the quotient $\omega_2/\omega_1$ indeed lies in the standard fundamental domain if $\lambda$ is restricted to  $ \mathcal{F}$. This latter fact is of course known, but we couldn't find a suitable reference and so have included a proof. In what follows  we first assume that $\lambda \in \Gamma$.  \\

We recall that the $j$-invariant of $E_\lambda $ is
\begin{align}
j = 256\frac{(\lambda^2 + \lambda -1)^3}{(\lambda-1)^2\lambda^2}.\label{Jinvariant}
\end{align}
The discriminant $D$ of $E_\lambda$ is equal to $ g_2^3 - 27g_3^2 = 16\lambda^2(1-\lambda)^2$ and it is well-known  that  the following relation holds
\begin{align} \label{Disc}
D = \left(\frac{2\pi}{\omega}\right)^{12}\Delta(\tau)
\end{align}
where
\begin{align*}
\Delta(\tau) = q\prod_{n=1}^\infty(1-q^n)^{24}
\end{align*}
with
\[
 q =\exp(2\pi i \tau).
  \]
Here $\omega, \tau$ are such that $\omega, \omega\tau$ span the lattice of $E_\lambda$ and  $\tau$ lies in the standard fundamental domain of the action of $SL_2(\Z)$, so $|\tau| \geq 1$ and $|\Re(\tau)| \leq \frac12$ and in particular $\Im(\tau) \geq \sqrt{3}/2$.  As $|A|$ is invariant under linear changes of $(\omega_1, \omega_2)$ by $GL_2(\Z)$  we can replace $(\omega_1,\omega_2)$ by the pair $(\omega,\omega\tau)$ and we see that
\begin{align}
|A| = 2|\omega|^2|\Im(\tau)|. \label{area}
\end{align}

Now we are ready to prove the following Lemma.

\begin{lem}\label{lemarea} If $\lambda \in \Gamma$ then
\begin{align*}
|A| \geq \max \left\{ \frac 4{\pi}(\log \max\{|\lambda|^{-1}, |1-\lambda|^{-1}\} -\log(11)), 2\sqrt{3}\right\}.
\end{align*}
\end{lem}

\begin{proof}
We first prove a lower bound for $|\omega|$. In view of the equality (\ref{Disc}) it is enough to prove a lower bound for $|\Delta/D|$. \\

The product $\prod_{n=1}^{\infty}(1 -q^n)^{24}$ is bounded uniformly from below as follows.
\begin{align}
|\prod_{n=1}^{\infty}(1 -q^n)^{24} | \geq \prod_{n=1}^{\infty}( 1 - \exp(-n\pi\sqrt{3}))^{24} \geq 9/10,\label{910}
\end{align}
where for the last inequality one could use standard estimates or a simple numerical computation. We also need the following inequality
\begin{align}
\frac{|q|^{-1}}{2080}\leq \max\{1, |j(\tau)|\} \leq 2080|q|^{-1}. \label{maxj}
\end{align}
which can be derived from the Fourier expansion of the $j$-function \cite[Lemma 1]{BMZ}. From (\ref{maxj}) we see that
\begin{align}
|q| \geq \frac1{\max\{1, |j(\tau)|\}2080}.
\end{align}

This with (\ref{Disc}) and (\ref{910}) leads to the lower bound $|\omega|^2 \geq 6/\delta$, where
\begin{align*}
\delta = \max\{1, 256\frac{|\lambda^2 - \lambda +1|^3}{|\lambda(\lambda -1)|^2}\}^{1/6}|\lambda(1-\lambda)|^{1/3}.
\end{align*}
%and finally
%$$A \geq \frac{11}{\max\{1, 256\frac{|\lambda^2 - \lambda +1|^3}{|\lambda(\lambda -1)|^2}\}^{1/6}|\lambda(1-\lambda)|^{1/3}}.$$

If the maximum on the right hand side of the above equality is attained at 1, then $\delta$  is equal to $|\lambda|^{\frac13}|1-\lambda|^{\frac13} \leq 1$. If the maximum is not attained at 1 then
\begin{align*}
\delta= 2^{\frac43}|\lambda^2 - \lambda +1|^{\frac12} =2^{\frac43}|(\lambda-\zeta)(\lambda-\overline{\zeta})|^{\frac12}\leq 2^{\frac43}\leq 3
\end{align*}
where $\zeta=\frac12 +i\sqrt{3}/2$. For the first inequality we used the fact that if we write $\lambda = r+it$ for real $r,t$ then $|r+it -\zeta| =|r-1/2 +i(t-\sqrt{3}/2)|$ is maximal for varying $r$ if $\lambda$ lies on the boundary of $\Gamma$ and the same holds for $\overline{\zeta}$. So we may assume  that $|\lambda|=1$ and $t^2= 1-r^2$. A short calculation shows that then $|(\lambda-\zeta)(\lambda-\overline{\zeta})|^2 =4(r-\frac12)^2 $ which is maximal at $r=1$ (as $r$ now lies between $1/2$ and 1).
Thus
\begin{align}
|\omega|^2 \geq 2\label{1}
\end{align}
and we can already deduce from (\ref{area}) that
\begin{align}
|A| \geq 2\sqrt{3}. \label{area2}
\end{align}
However we can go a little further and note that from  (\ref{maxj}) we have
\begin{align}
\Im(\tau) \geq \frac1{2 \pi}( \log \max\{1,|j|\} - \log2080).\label{imtau}
\end{align}
Now we will bound $ \log \max\{1,|j|\}$ from below. \\

We denote by $\zeta = \frac12 + \frac{\sqrt{-3}}2$ a root of $\lambda^2 - \lambda +1$. We can check that for $\lambda \in \Gamma$ and $\min\{|\lambda|, |1-\lambda|\} \leq \frac12$ we have
\begin{align*}
|\lambda^2 - \lambda +1 | & = |\lambda - \zeta||\lambda - \overline{\zeta}| \geq \sqrt{3}/4
\end{align*}
and so
\begin{align*}
|j|  &  \geq 18\max\{|\lambda|^{-2},|1 -\lambda|^{-2}\}, \text{ for } \min\{|\lambda|, |1-\lambda|\} \leq \frac12.
\end{align*}

With (\ref{imtau}) we get that
\begin{align*}
\Im(\tau) \geq \frac1{\pi}(\log \max\{|\lambda|^{-1}, |1-\lambda|^{-1}\} -\log(11)) \text{ for } \min\{|\lambda|, |1-\lambda|\} \leq \frac12,
\end{align*}
which, with (\ref{area}) and (\ref{area2}), completes the proof.

\end{proof}

We now prove some standard facts about $\mathcal{F}$. As mentioned above, these are well-known but we include proofs as we were unable to find a reference. \\

It is well-known that the symmetric group $S_3$ acts on $\C\setminus \{0,1\}$ by the transformations generated by $\lambda \rightarrow 1/\lambda$, $\lambda \rightarrow 1-\lambda$. The orbit $O_\lambda$ of an element $\lambda \in \C\setminus \{0,1\}$ is given by
\begin{align*}
O_\lambda = \{\lambda, 1/\lambda, 1-\lambda, 1/(1-\lambda), \lambda/(\lambda-1), (\lambda-1)/\lambda\}.
\end{align*}
We first observe that $\mathcal{F}\setminus \mathcal{A}$,where $\mathcal{A}$ is the set
\begin{align*}
\mathcal{A} = \{\lambda \in \C \setminus \{0,1\}; |1-\lambda| = 1, \Re(\lambda) \leq \frac12, \Im(\lambda) <0\}\cup\{ \lambda \in \C: \Re(\lambda)=\frac12, \Im(\lambda)<0\}
\end{align*}
is a fundamental domain for the action of $S_3$. Below we write $\mathcal{A}^*$ for the image of $\mathcal{A}$ under complex conjugation.
\begin{lem}
The sets $\mathcal{F}\setminus \mathcal{A}$ and $\mathcal{F}\setminus \mathcal{A}^*$  are  fundamental domains for the action of $S_3$ on $\C\setminus\{0,1\}$.
\end{lem}
\begin{proof}
We already noted that $\mathcal{F}$ contains at least one element of each orbit of $S_3$. The transformation that sends $\lambda$ to $\lambda/(\lambda -1)$ acts like complex conjugation on the circle $|1-\lambda| = 1$ so also $\mathcal{F}\setminus \mathcal{A}$ and $\mathcal{F}\setminus \mathcal{A}^*$ contain a fundamental domain of $S_3$.  Now it remains to check that  $\mathcal{F}\setminus \mathcal{A}$ does not contain two distinct elements of one orbit. We leave the details to the reader.
%Before we proceed with checking each transformation we note that the orbit of $1/2+i\sqrt{3}/2$ is $O_{1/2+i\sqrt{3}/2} = \{1/2+i\sqrt{3}/2, 1/2-i\sqrt{3}/2\}$. First, if $\lambda$ and $1/\lambda$ both lie in $\mathcal{F}$  then $|\lambda| = 1$ and so $\lambda \in O_{1/2+i\sqrt{3}/2}$. If $\lambda$ and $1-\lambda$ are both in $\mathcal{F}$ then $\Re(\lambda) = \frac12$ and on such elements, $\lambda \rightarrow 1-\lambda$ acts like complex conjugation. If both $\lambda$ and $1/(1-\lambda)$ are in $\mathcal{F}$ then $|1-\lambda| = 1$ and so $1/(1-\lambda) = 1- \overline{\lambda}$ and so also $\Re(\lambda) = \frac12$. Hence again $\lambda \in O_{1/2+i\sqrt{3}/2}$. If $\lambda \in \mathcal{F}$ then $|\lambda/(\lambda -1)| \geq 1$ so if $\lambda/(\lambda -1) \in \mathcal{F}$ then $\lambda/(\lambda -1) \in O_{1/2+i\sqrt{3}/2} $ and so $\lambda \in O_{1/2+i\sqrt{3}/2} $. Finally if $\lambda \in \mathcal{F}$ then $|1-(\lambda -1)/\lambda|  = |1/(1-\lambda)|\geq 1$ so if $(\lambda -1)/\lambda \in \mathcal{F}$ as well then $|1-\lambda| = 1$. So $\lambda \in \mathcal{A}\cup \mathcal{A}^*$. We already noted that the transformation $\lambda \rightarrow (\lambda -1)/\lambda$ acts like complex conjugation on the circle $|1-\lambda| = 1$.
\end{proof}

We write $\mathcal{S}$ for the (closure of the) standard fundamental domain in the upper half plane:
\begin{align*}
\mathcal{S}=\{\tau \in \mathbb{H}; |\Re(\tau)|\leq \frac12, |\tau|\geq 1\}.
\end{align*}

\begin{lem}\label{lemlowerbounds} The set $\mathcal{S}$ is the image of $\mathcal{F}$ under the function $\omega_2/\omega_1$. In particular, for $\lambda \in \mathcal{F}$, we have $|\Re(\omega_2/\omega_1)|\leq 1/2$ and $|\omega_2/\omega_1| \geq 1$. Further for $\lambda \in \mathcal{F}$ we have $\min\{|\omega_1|, |\omega_2|\}\geq 1$.
\end{lem}
\begin{proof}
The strategy is to show that the boundary of $\mathcal{F}$ maps surjectively to the boundary of $\mathcal{S}$ under the map $\omega_2/\omega_1$. This shows with the previous Lemma that no point in the interior of $\mathcal{F}$ maps to the boundary of $\mathcal{S}$. We also show that at least one point in the interior of $\mathcal{F}$ maps to the interior of $\mathcal{S}$ and then conclude with the intermediate value theorem. \\

To begin, note that $\omega_2/\omega_1(\frac12) = i$. From the series expansion of $F(\lambda)$ we see that $\omega_2/(i\omega_1)$ is a real increasing function on the interval $(0,\frac12]$ approaching $+\infty$ as $\lambda$ approaches 0. Hence if $\lambda \in (0,\frac12)$ then $\omega_2/\omega_1(\lambda)$ lies in the interior of $\mathcal{S}$.  We also observe that $\omega_2(1-\lambda) = i\omega_1(\lambda)$ and since the series defining $F(\lambda)$ has real coefficients $\overline{\omega_2}(\lambda) = -\omega_2(\overline{\lambda})$ and $ \overline{\omega_1}(\lambda) = \omega_1(\overline{\lambda})$. So for $\Re(\lambda) = \frac12$  we have $|\omega_2/\omega_1| = 1$. From (\ref{Jinvariant}) we read off that $j=0$ for $\lambda \in \{1/2 +i\sqrt{3}/2,1/2 -i\sqrt{3}/2 \}$ and as $j(1/2 \pm i\sqrt{3}/2) = 0$ we deduce that $\omega_2/\omega_1(1/2 \pm i\sqrt{3}/2) \in\{1/2 +i\sqrt{3}/2,1/2 -i\sqrt{3}/2 \} $. We also deduce  from the previous Lemma that $\omega_2/\omega_1$ is injective on the line $\{\lambda \in \mathcal{F}; \Re(\lambda) = \frac12\}$. Hence $\omega_2/\omega_1$ maps the line $\{\lambda \in \mathcal{F}; \Re(\lambda) = \frac12\}$ bijectively to the arc $\{\tau;|\tau| =1, |\Re(\tau)| \leq\frac12\} $. \\

We set $\omega_2/\omega_1(1/2 +i\sqrt{3}/2) = \tau_0$ and $\omega_2/\omega_1(1/2 -i\sqrt{3}/2) = \tau_1$, where $\tau_0,\tau_1$ are distinct elements of $\{1/2 +i\sqrt{3}/2,1/2 -i\sqrt{3}/2 \}$. We will show now that the image of $\mathcal{A}$ is the line $L_0 = \{\tau;~\Re(\tau) = \Re(\tau_0),|\tau|\geq 1\} $ and of $\mathcal{A}^*$ is the line $L_1 = \{\tau;\Re(\tau) = \Re(\tau_1),|\tau|\geq 1\}$.
We first remark that $j(\tau)$ is real for $\tau \in \mathcal{S}$ if and only if $\Re(\tau) \in \frac12\Z$  or $|\tau|=1$. As complex conjugation acts like a transformation of $S_3$ on the circle $|1-\lambda| = 1$, $j$ is real for $\lambda$ in $\mathcal{A}\cup\mathcal{A}^*$. We argue only for $\mathcal{A}$. The arguments for $\mathcal{A}^*$ are exactly analogous. \\

We have that $j(\tau_0)= j'(\tau_0)= j''(\tau_0) = 0$ and $ j'''(\tau_0) \neq 0$. Hence the set $\Im(j) =0$ locally at $\tau_0$ consists of the image of  at most three simple real analytic curves intersecting in $\tau_0$. One of these curves is $\{\Re(\tau) = \Re(\tau_0)\}$ and the other two are $\{|\tau| = 1\} $ and $\{|\tau \pm 1| = 1\}$ where the sign $\pm$ depends on the value of $\tau_0$. It is well-known that the derivative of $j$ vanishes only on the sets $\Z + \tau_0 $ and $\Z + i$. In particular the set $\Im(j) = 0$ is locally the image of an analytic curve anywhere else (does not ``branch''). Thus if for some $\lambda \in \mathcal{A}$, $\omega_2/\omega_1(\lambda)$ lies on one of those curves and satisfies $\Im(\omega_2/\omega_1(\lambda)) <\sqrt{3}/2$ then since $\mathcal{F}\setminus \mathcal{A}^*$ is a fundamental domain for $S_3$ all $\lambda \in \mathcal{A}$ satisfy $\Im(\omega_2/\omega_1(\lambda)) < \sqrt{3}/2$. From the singular expansion of $\omega_2/\omega_1$ we can read off that $\Im(\omega_2/\omega_1)$ tends to infinity as $\lambda$ approaches 0 so we get a contradiction. Further, again as $\mathcal{F}\setminus \mathcal{A}^*$ is a fundamental domain $\omega_2/\omega_1(\lambda)$ also can not lie on any of the circles of radius 1 centred at an integer and satisfy $\Im(\omega_2/\omega_1(\lambda)) > \sqrt{3}/2$. Thus $\omega_2/\omega_1(\lambda)$ lies on $L_0$ for all $\lambda \in \mathcal{A}$. As $\Im(\omega_2/\omega_1(\lambda))$ tends to infinity as $\lambda$ approaches 0  $\omega_2/\omega_1$ restricted to $\mathcal{A}$ surjects onto $L_0$. Now if there were some $\lambda$ in the interior of $\mathcal{F}$ that does not map to $\mathcal{S}$, then by the intermediate value theorem there would be another $\lambda$ in the interiour that maps to the boundary of $\mathcal{S}$. This concludes the proof of the first part of the Lemma. \\

From the proof of Lemma \ref{lemarea} we can read off that $|\omega_1|\geq 1$ and as $|\omega_2/\omega_1|\geq1$ the second part of the present Lemma follows.
\end{proof}

\section{Reduction to the propositions}
In this section we deduce Theorem \ref{defp}, \ref{defz} and \ref{defs} from the propositions. The proof goes by constructing the definitions explicitly. \\

We start with $\wp$. Fix $\lambda\in \F$, and define $z(\xi)=z(\lambda,\xi)$ on $X_\lambda$ as in  section \ref{analyticcontinuation}. We assume that the imaginary part of $\lambda $ is not negative. The definition in the other case is completely analogous. Let $\Omega_\lambda$ be the lattice generated by $\omega_1(\lambda)$ and $\omega_2(\lambda)$ and $\frak{F}_\lambda$ be the fundamental domain $\frak{F}_{\Omega_\lambda}$ spanned by those two periods. Let $\wp$ be the $\wp$-function associated to this lattice.

Let $V_1$ be the half-plane north of the line $Im(\xi)=Im(\lambda)$. Let $V_2$ and $V_3$ be the pieces of the strip $0<Im(\xi)<Im(\lambda)$ lying  west and  east of the line $L_\lambda$, respectively. And let $V_4$ be the half-plane with negative imaginary part. Let $V_5$ be the horizontal line extending west from $\lambda$, and $V_6$ be the horizontal line extending east from $\lambda$. And let $V_7,V_8$ and $V_9$ be the three lines removed to make $X_\lambda$. Finally let $V_{10}$ be the interval $(0,1)$. On each of these sets we consider $z$ as above, and also its other branch $-z$. We denote the real and imaginary parts of the two branches by $u_-,v_-$ and $u_+,v_+$ (we suppress the dependence on the domain here). Below we write these as functions of one complex number rather than two reals, except in the following lemma.\\

\begin{lemma} Fix $i$. On the domain $V_i$ the real functions $u_-,v_-$ collectively pfaffian of order $7$ and degree $(9,1)$. The same statement holds for $u_+,v_+$.\end{lemma}
\begin{proof} This lemma is due to Macintyre \cite{Macintyre}. We give some of the details, following \cite{JTzeta}. We just write $u,v$ with the choice of branch fixed. For $\xi \in V_i$ write $\xi_r$ and $\xi_i$ for the real and imaginary parts of $\xi$, respectively. Then we have
\begin{align*}
\d{u}{\xi_r} & =\d{v}{\xi_i} =\frac{ \Re \left( \sqrt{ g_\Omega (\xi) }\right)}{|g_\Omega (\xi)|}\\
\d{u}{\xi_i} & =-\d{v}{\xi_r} =\frac{ \Im \left( \sqrt{ g_\Omega (\xi) }\right)}{|g_\Omega (\xi)|}
\end{align*}
where $g_\Omega(\xi)= \xi(\xi-1)(\xi-\lambda)$. Let $A_\Omega$ and $B_\Omega$ be the real and imaginary parts of the polynomial $g_\Omega$. Then
\begin{eqnarray*}
\Re \sqrt{ g_\Omega} =\frac{1}{\sqrt2} \sqrt{ \sqrt{A_\Omega^2 +B_\Omega^2}+A_\Omega }\\
\Im \sqrt{ g_\Omega} =\frac{1}{\sqrt2} \sqrt{ \sqrt{A_\Omega^2 +B_\Omega^2}-A_\Omega }.
\end{eqnarray*}
Here the inner square root is taken positive, and the outer one has the sign that makes the equations for the partials above hold. We then take
\begin{eqnarray*}
f_1 &=& \frac{1}{\sqrt{A_\Omega^2+B_\Omega^2}} \\
f_2 &=& \frac{1}{\sqrt{\sqrt{A_\Omega^2+B_\Omega^2}+A_\Omega}} \\
f_3 &=& \frac{1}{\sqrt{\sqrt{A_\Omega^2+B_\Omega^2}-A_\Omega}}\\
f_4 &=&\sqrt{\sqrt{A_\Omega^2+B_\Omega^2}+A_\Omega}\\
f_5 &=&\sqrt{\sqrt{A_\Omega^2+B_\Omega^2}-A_\Omega}.
\end{eqnarray*}
And then with $f_6=u$ and $f_7=v$ we have a pfaffian chain of order $7$ and degree $(9,1)$.
\end{proof}

Here is our first main result.
\begin{thm} Suppose that $\Omega$ is a lattice in the complex numbers, and that $\wp$ is the associated $\wp$-function. Then on a fundamental domain $\frak{F}_\Omega$ for $\Omega$, the graph of $\wp|_{\frak{F}_\Omega}$ is a piecewise semipfaffian set of format $(7,9,1,4,144503,2)$.
\end{thm}
\begin{proof} Let $\lambda \in \mathcal{F}$ be such that $c\Omega=\Omega_\lambda$ for some complex number $c$, where $\Omega_\lambda$ is the lattice generated by $\omega_1(\lambda)$ and $\omega_2(\lambda)$. Let $\wp_\lambda$ be the associated $\wp$-function. Since
\[
( z,\xi ) \in \text{graph}(\wp_{\Omega}) \text{ if and only if } (cz,c^{-2}\xi)\in \text{graph}(\wp_{\lambda})
\]
it is sufficient to show that the graph of $\wp_\lambda$, restricted to the fundamental domain $\frak{F}$ given by $\omega_1(\lambda)$ and $\omega_2(\lambda)$, is a piecewise semipfaffian set of the format claimed.

For $j=1,\ldots,10$, we let $V_j^*=\{ z- \frac13(\lambda+1) : z\in V_j\}$ (note that this is still a simple domain). We define a chain on $V_j^*$ by composing the chain from the previous lemma with $z+\frac13(\lambda+1)$. This doesn't change complexities. In particular, we will write $u_{-}^*(z)$ for the function $u_-(z+\frac13(\lambda+1))$ and similarly for the other branch, and for the branches of the imaginary part. For $j=1,\ldots,10$ and integers $m,n$ with absolute value at most $42$ we put
\begin{eqnarray*}
Y^{\pm}_{\wp,j,m,n}= \Big\{  (x,y,f_{\wp},g_{\wp}) \in \frak{F}_\Omega \times V_j^*:
 u_{\pm}^*(f_{\wp}+ig_{\wp})+iv_{\pm}^*(f_{\wp}+ig_{\wp})=  x+iy+m\omega_1+n\omega_2 \Big\}.
\end{eqnarray*}
We then have
\[
\text{graph}(\wp_\lambda|_{\frak{F}})= \bigcup_{j=1,\ldots,10,|m|,|n|\le 42} Y^{\pm}_{\wp,j,m,n}\]
\[ \cup\left\{\left(\frac{\omega_1}2, \frac{\lambda-2}3\right),\left(\frac{\omega_2}2, -\frac{\lambda+1}3\right),\left(\frac{\omega_1 + \omega_2}2, \frac{1 - 2\lambda}3\right)\right\}.
\]
That this union does indeed give all the graph follows from Proposition \ref{propbetti}. And it is easy to see that the format is that claimed.
\end{proof}

\begin{thm} Suppose that $\Omega$ is a lattice in the complex numbers, and that $\zeta$ is the associated $\zeta$-function. Then on a fundamental domain $\frak{F}_\Omega$ for $\Omega$, the graph of $\zeta|_{\frak{F}_\Omega}$ is a piecewise subpfaffian set of format $(9,9,1,6,144503,4)$.
\end{thm}
\begin{proof} As before, we take $\lambda \in \mathcal{F}$ such that $c\Omega=\Omega_\lambda$ for some complex number $c$, where $\Omega_\lambda$ is the lattice generated by $\omega_1(\lambda)$ and $\omega_2(\lambda)$. Then $\zeta_{\Omega}(z) = c \zeta_{\lambda}(cz)$ for $\zeta_{\lambda}$ the $\zeta$-function associated to $\Omega_\lambda$. And so it is enough to show the result for $\zeta_\lambda$.

For each $j=1,\ldots, 10$ pick a point $a_j\in V_j$. Define
\[
\tilde{G}_{j}(\hat{X})= \int_{a_j}^{\hat{X}} \frac{X-\frac13(\lambda+1)}{2\sqrt{X(X-1)(X-\lambda)}}dX,
\]
a function on $V_j$. Then we have
\[
\tilde{G}_{j}(\hat{X})=-\zeta_\lambda(z(\hat{X}))+\zeta_\lambda(z(a_j)).
\]
Adding the real and imaginary parts of $\tilde{G}_{j}$ to the chain from the lemma above we get a chain of length $9$ and degree still $(9,1)$. We then shift again, defining
\[
G_{j}(\hat{X})=-\tilde{G}_{j}(\hat{X}+\frac13(\lambda+1))+\zeta_\lambda(z(a_j)).
\]
on $V_j^*$. And then we have
\[
G_{j}\left(\wp_\lambda(z)\right)=\zeta_\lambda(z),
\]
with real and imaginary parts occurring in a chain of length 9, together with $u_-^*,v_-^*$.
Let
\[
Y^{\pm}_{\zeta, j,m,n}=\{(x,y,f_\wp,g_\wp,f_\zeta,g_\zeta)\in \mathfrak{F}_\Omega\times V_j\times \R^2 : (x,y,f_\wp,g_\wp)\in Y^\pm_{\wp,j,m,n} \]
\[
\text{ and } f_\zeta+ig_\zeta= G_{j}(f_\wp,g_\wp)-\eta(m\omega_1+n\omega_2)\}
\]
Then
\begin{eqnarray*}
\text{graph}(\zeta_\lambda|_{\frak{F}_\Omega})= \pi \Bigg(\bigcup_{j=1,\ldots,10, |m|,|n|\le 42} Y^\pm_{\zeta,j,m,n}  &\\
\cup\left\{\left(\frac{\omega_1}2,1, \frac{\eta_1}2\right),\left(\frac{\omega_2}2, 0, \frac{\eta_2}2\right), \left(\frac{\omega_1 + \omega_2}2,\lambda, \frac{\eta_1+\eta_2}2\right) \right\}\Bigg)
\end{eqnarray*}
 where $\pi$ is the projection which omits the middle two coordinates (that is, the $\wp$-coordinates).\\

\end{proof}
Now we turn to $\varphi_\lambda$. We define
\begin{align*}
\hat{H}_{V_j}(\hat{X}) = -\frac{\eta_1}{\omega_1}z(\hat{X}) + \frac{\pi i}{\omega_1} + G_{V_j} (\hat{X})
\end{align*}
on $V_j$ and then define the function $\mathcal{L}_{V_j}$ by
\begin{align*}
\mathcal{L}_{V_j}(\xi) =-\int_{1}^{\xi}\frac{\hat{H}_{V_j}d\hat{X}}{2\sqrt{\hat{X}(\hat{X} - 1)(\hat{X} - \lambda)}}.
\end{align*}
If we define the real and imaginary part of $\mathcal{L}_{V_j}$ to be $f_{10},f_{11}$ and them to the chain formed by $f_1,\dots, f_9$ we get a chain of order 11 and degree still $(9,1)$. \\

We also need to define a chain for the exponential function restricted to  the set $\frak{F}_e$ of complex numbers whose imaginary part $\Im$ satisfies $-\pi\leq \Im < \pi$. We write
\begin{align*}
\exp(x+iy) = \exp(x)(\cos(y) +i\sin(y))
\end{align*}
and note that the functions $\exp(x),\tan(y/3),\cos(y/3)$ form a chain of order 3 and degree 2 on the interval $(-3\pi/2,3\pi/2)$. And with that chain, $\sin(y/3)$ has degree $(2,2)$. Then using the fact that $\sin(y)$ and $\cos(y)$ are polynomials of $\sin(y/3)$ and $\cos(y/3)$, respectively, of degree 3 we find that the real and imaginary part of $\exp(x+iy)$ is a Pfaffian function of order 3 and degree $(2,6)$ on the simple domain $\R\times [-\pi,\pi)$. %As $\cos(y)$ is a polynomial of degree 3 in $\cos(y)$ we deduce that $\exp(x)\cos(y)$ and $\exp(x)\sin(y)$ are Pfaffian functions of order 3 and degree $(2,4)$ on the simple domain $\frak{F}_e$.

%define $f_{12}=\exp(x), f_{13} =\cos(y/3), f_{14}=\cos^2(y/6), f_{15}=\sin(y)$.  Then $f_1,\dots, f_{15}$ form a chain of degree $(9,1)$ %and order 15. We will use the fact that $\cos(y/3)$ is a polynomial of degree 3 in $\cos(y)$.\\

From the discussions in section \ref{analyticcontinuation}, in particular (\ref{diffphi}), (\ref{phi}), we have
\begin{align*}
\varphi_\lambda(\omega_1/2)\exp(\mathcal{L}_{V_j}(\xi)) = \varphi_\lambda(z(\xi)).
\end{align*}
However as we have to deal with translations by periods we also note that if $\tilde{z}$ is the translate of $z$ that lies in the fundamental domain $\frak{F}$ then by Proposition \ref{propbetti} $z = \tilde{z} +m\omega_1 + n\omega_2$ where $|m|,|n|\leq 42$ and
\begin{align*}
(-1)^n\exp(\psi_n(\tilde{z}))\varphi_\lambda(\tilde{z}) = \varphi_\lambda(z).
\end{align*}
where
\begin{align*}
\psi_n = -2\pi i n\tilde{z}/\omega_1 -\pi i n(n-1)\omega_2/\omega_1
\end{align*}
if $n \geq 0$ and
\begin{align*}
\psi_n = 2\pi i n\tilde{z}/\omega_1 -\pi i n(n+1)\omega_2/\omega_1
\end{align*}
otherwise.
Using Lemma \ref{lemlowerbounds} and Proposition \ref{propbetti} we compute that
\begin{align*}
|\Im(\psi_n(\tilde{z}))/(2\pi)| \leq 515.
\end{align*}
%and so we finally define
%\begin{align*}
%H_{V_j}= \varphi_\lambda(\omega_1/2)\exp(\mathcal{L}_{V_j}(\xi)).
%\end{align*}
Now we continue with the definition of the graph of $\varphi_\lambda$.  We define
\begin{align*}
Y_{\varphi,j,m,n,k,l}^{\pm} =\{(x,y,x_{\varphi}, y_{\varphi},f_{\mathcal{L}}, g_{\mathcal{L}}, f_{\wp}, g_{\wp}, f_{\varphi}, g_{\varphi}) \in \mathfrak{F}_\Omega\times \mathfrak{F}_e^2\times V_j \times \R^2 ;\\
 (x,y,f_{\wp}, g_{\wp}) \in Y^{\pm}_{\wp,j,m,n},
 f_{\mathcal{L}}+ig_{\mathcal{L}} + 2k\pi i = \mathcal{L}_{V_j}(f_\wp,g_\wp),
x_{\varphi} + iy_{\varphi} +2l\pi i = \psi_n(x+iy) \\  \text{ and }
(-1)^n\exp(x_{\varphi} + iy_{\varphi})(f_{\varphi}+ig_{\varphi}) = \varphi_\lambda(\omega_1/2)\exp( f_{\mathcal{L}}+ig_{\mathcal{L}} )\}.
\end{align*}

Then by Proposition \ref{propsigma}, we have that
\begin{eqnarray*}
\text{graph}(\varphi_\lambda|_{\frak{F}_\Omega})=  \pi_{x,y,f_{\varphi},g_{\varphi}} \left(\bigcup_{j=1,\ldots,10, |m|,|n|\leq 42,  |k|\leq 384, |l|\leq 515} Y^\pm_{\varphi ,j,m,n,k,l} \right) &\\
\cup\left\{\left(\frac{\omega_1}2, \varphi_\lambda\left(\frac{\omega_1}2\right)\right),\left(\frac{\omega_2}2, \varphi_\lambda\left(\frac{\omega_2}2\right)\right), \left(\frac{\omega_1 + \omega_2}2, \varphi_\lambda\left(\frac{\omega_1 + \omega_2}2\right)\right)  \right\}
\end{eqnarray*}
is a piecewise subpfaffian set of format $(17,9,6,10,114565235503,8 )$.
For general $\Omega$ we pick $\lambda \in \F$ such that $c\Omega = \Omega_\lambda$. We have
\[
\varphi_{\Omega}(z) = c^{-1} \varphi_{\lambda}(cz).
\]
So we can define $\varphi_{\Omega}$ on $\frak{F}_\Omega$, by
\[
( z,\xi ) \in \text{graph}(\varphi_{\Omega}) \text{ if and only if } (cz,c\xi)\in \text{graph}(\varphi_{\lambda}).
\]

\section{Bounding the numerator from above}

In this section we establish a logarithmic  upper bound for the ``numerators''  $B_1= Ab_1 $ and $ B_2= Ab_2$ of the Betti-coordinates on $(-\infty,0]\cup L_\lambda \cup [1,\infty) $.  We recall that they are given by
\begin{align}
B_1 & = \overline{\omega_2}z - \omega_2\overline{z} \label{B1}\\
B_2 & = \omega_1\overline{z}-\overline{\omega_1}z. \label{B2}
\end{align}
It would be relatively straightforward to  get  a bound of the order of $\log^2$  by estimating each term  but in order to get a bound with the right growth we have to aim for some cancellation in the sum. Because of the way we've set things up, the main obstacle lies in estimating $B_1$ on $(-\infty,0]$. Before we address that  problem we prove some inequalities that can be achieved by rather standard estimates but are nevertheless important for us. We define $z = z (\lambda,\xi)$ on $\Gamma\times (-\infty,0]$ as in the introduction by (\ref{inverse}).

\begin{lem} \label{estimatesintegrals} If $(\lambda, \xi)$ is in $\Gamma\times (-\infty,0]$  then
\begin{align*}
|\omega_2|  \leq  \sqrt{2}\log(|\lambda|^{-1}) + 5,\end{align*} and \begin{align*}
|z|  \leq \log(|\lambda|^{-1})+ \frac52.
\end{align*}
And if $\lambda$ is in $\mathcal{F}$ then
\begin{align*}
|\omega_1| \leq 5.
\end{align*}
\end{lem}
\begin{proof}
We set $X = -t$ and note that
as $\Re(\lambda) \geq 0$ standard inequalities yield
\begin{align}
|t +\lambda| \geq \frac12(t +|\lambda|),  \label{ineq}
\end{align}
for $ \lambda \in \Gamma$. We use the integral expression (\ref{ellintegral}) for $\omega_2$ and get that
\begin{align*}
|\omega_2| & \leq  \sqrt{2}\left( \int_0^1\frac{dt}{\sqrt{t(t+1)(t + |\lambda|)}} + \int_{1}^{\infty}\frac{dt}{\sqrt{t(t+1)(t+|\lambda|)}}\right) \\
& \leq \sqrt{2}\left(\int_0^1\frac{dt}{\sqrt{t(t + |\lambda|)}} + \int_1^{\infty}\frac{dt}{t\sqrt{t+1}} \right) \leq \sqrt{2}\left(2\log(\sqrt{2} +1) +177/100 +\log(|\lambda|^{-1})\right)\\
& \leq \sqrt{2}\log(|\lambda|^{-1}) + 5
\end{align*}
where for the second inequality we have used that
\begin{align}
 \frac12\int\frac{dX}{\sqrt{X( X + |\lambda|)}} = \log \left(\sqrt{X + |\lambda|} +\sqrt{X} \right).\label{primitive}
\end{align}
We can perform the same estimates for $z$ but note that the integrand for $z$ is one-half of that for $\omega_2$. This provides us with the inequality for $z$.\\

Now it remains to prove the inequality for $\omega_1$. We again use the integral expression (\ref{ellintegral}). There we set
$X = 1 +t$ and note that
\begin{align*}
|X(X-1)(X - \lambda)|\geq \max\{\frac12|t|,|t|^3\},  \text{ for } \lambda \in \mathcal{F}.
\end{align*}
Thus we have
\begin{align*}
|\omega_1| \leq \sqrt{2}\int_{0}^1\frac{dt}{\sqrt{t}} + \int_{1}^\infty\frac{dt}{t^{\frac32}} \leq 5.
\end{align*}
\end{proof}
In order to estimate $B_1$ we develop $z$ as a series in $\lambda$ at 0. We first write
\begin{align*}
z^{(n)}(\xi) = \left(\frac{1}2\right)_{n} \int_{\xi}^{-\infty}\frac{dX}{2X^{n+1}\sqrt{X-1}}
\end{align*}
so that
\begin{align}
z(\lambda,\xi) = \sum_{n=0}^{\infty}\frac {z^{(n)}(\xi)}{n!}\lambda^n \label{taylorz}
\end{align}
when this series converges in a neighbourhood of 0.  We can then write $z = z^{(0)} + \tilde{z}$. We have a similar expansion for $\omega_2$, given by
\begin{align}
\omega_2(\lambda) = i\frac{\omega_1}{\pi}\log(\lambda) +u(\lambda) \label{omega2}
\end{align}
where $u$ is an analytic function at 0 and $\log$ is the canonical branch of the logarithm. We can give a series expansion for $u$  \cite[p.299]{Whittaker-Watson}
 \begin{align}
u = i\sum_{n=0}^{\infty}\frac{(\frac12)_n^2(4\log2 - 4\gamma_n)}{(n!)^2}\lambda^n  \label{u}
\end{align}
where $\gamma_n = 1 -\frac12 +\dots -\frac1{2n}, (\gamma_0 = 0)$. We also quickly recall that the Taylor expansion of $\omega_1$ at 0 is
\begin{align*}
\omega_1 = \pi\sum_{n=0}^\infty\frac{(\frac12)_n^2}{n!^2}\lambda^n.
\end{align*}
We write $\omega_1 = \pi + \tilde{\omega}$. As we will be dealing with $\log$ as a real analytic function it is also convenient to set
\begin{align*}
\log(\lambda) = \log(|\lambda|) + i\arg(\lambda)
\end{align*}
where $-\frac{\pi}2 \leq \arg\leq \frac{\pi}2$ (as we are working on $\Gamma$).
The next Lemma, which is the main estimate of this section, deals with the main issue connected to the singularity at 0.

\begin{lem}\label{L1}
If $(\lambda, \xi) \in \mathcal{F}\times (-\infty,0]$ then
\begin{align*}
\max\{|B_1|, |B_2|\} \leq 14\log|\lambda|^{-1} + 36.
\end{align*}
\end{lem}
\begin{proof}
We first get the estimate for $B_2$ out of the way. It suffices to use the triangle inequality for the terms  in the expression (\ref{B2}) and plug in the estimates from Lemma \ref{estimatesintegrals}.\\

For $B_1$ we first address the convergence of (\ref{taylorz}). As
$(\frac12)_n/n! \leq 1$ we have
\begin{align*}
|z^{(n)}|/n! \leq |\int_{0}^{\infty}\frac{dt}{2(t-\xi)^{n+1}\sqrt{t-\xi +1}}| \leq |\int_{0}^{\infty}\frac{dt}{(t-\xi)^{n+1}}| \leq |\xi|^{-n}.
\end{align*}
Hence the series (\ref{taylorz}) converges absolutely whenever $|\xi| \geq C|\lambda|$ for some $C>1 $. We set $C=2$ and, until we explicitly say otherwise, assume that $|\lambda| \leq |\xi|/2$.  Thus we are now in the case in which $\xi$ is near to $-\infty$.  \\

Now we plug the expansions
\begin{align*}
z = z^{(0)}(\xi) + \tilde{z}, ~\omega_2 = i\log(|\lambda|)  -\arg(\lambda) + \frac{i\tilde{\omega}}{\pi}\log(\lambda) +u
\end{align*}
into (\ref{B1}).
This leads to $ B_1  = M +R$ with
\begin{align}
M & = -i(z^{(0)}  +\overline{z^{(0)}})(\xi)\log(|\lambda|) +(\overline{z^{(0)}}- z^{(0)})(\xi) \arg(\lambda). \label{expansion}
\end{align}
But we will be more precise in our definition of  the integral involved here. We set
$X= \xi - t$ and plugging this into the integral for $z^{(0)}$ we get
\begin{align}
z^{(0)} = i\int_{0}^{\infty}\frac{dt}{2(t-\xi)\sqrt{t +1 -\xi}}
\end{align}
where  we made the choice $\sqrt{-1} = i$. Note that the integral above is an honest real integral now and we see that
\begin{align*}
z^{(0)} + \overline{z^{(0)}} = 0
\end{align*}
 (independently of our choice for the square-root).

This is the cancellation that we hoped for. We now get back to estimating.

First $z^{(0)}(\xi)$, which appears in $M$. If $|\xi| \geq 1$ then
\begin{align*}
|z^{(0)}(\xi)| = \int_{0}^{\infty}\frac{dt}{2(t-\xi)\sqrt{t +1 -\xi}} \leq \int_{0}^{\infty}\frac{dt}{2(t-\xi)^{\frac32}} \leq 1.
\end{align*}
When $|\xi| < 1$ we note that
\begin{align*}
|z^{(0)}(\xi)| =\int_{0}^{1}\frac{dt}{2(t-\xi)\sqrt{t +1 -\xi}} + \int_1^\infty\frac{dt}{2(t-\xi)\sqrt{t +1 -\xi}}
\end{align*}
and
\begin{align*}
\int_1^{\infty} \frac{dt}{2(t-\xi)\sqrt{t +1 -\xi}} \leq \int_1^{\infty}\frac{dt}{2(t-\xi)^{\frac32}} \leq 1
\end{align*}
while with a similar argument as above
\begin{align*}
 \int_0^1\frac{dt}{2(t-\xi)\sqrt{t +1 -\xi}} \leq \int_0^1\frac{dt}{2(t-\xi)} \leq \frac12\log(|\xi|^{-1}) + \log(2)/2.
\end{align*}
We deduce from the above that
\begin{align}
|z^{(0)}(\xi)| \leq
\begin{cases}  \frac12\log(|\xi|^{-1})| +1 &\mbox{if } \xi \leq 1 \label{z0}\\
\frac12 & \mbox{else}. \end{cases}
\end{align}
Now we are going to treat $R$. From now on we also assume that $|\lambda| \leq \frac 12$ unless we say otherwise. \\

Unravelling the terms of $R$ by first decomposing $\omega_1 =\pi +\tilde{\omega}$ collecting terms and then applying the triangle inequality  we find that
\begin{align}
|R| \leq 2\left|\left(\frac{i\tilde{\omega}}{\pi}\log(\lambda) +u\right)z^{(0)}\right| + 2|\omega_2\tilde{z}| \label{R}
\end{align}
We first estimate $u$ and note that $\frac 12 \leq \gamma_n \leq \log 2 $  so
\begin{align*}
4|  \log 2 -\gamma_n|\leq 4\log2 \leq 3.
\end{align*}
By majorising $u, \omega_1/\pi, \frac{\tilde{\omega}}{\pi \lambda }$ and $\tilde{z}$ by a geometric series we obtain
\begin{align*}
|u| \leq 3, ~\left|\frac{\tilde{\omega}}{\pi \lambda }\right| \leq 2,  ~ |\omega_1/\pi|\leq 2,~|\tilde{z}|\leq 1
\end{align*}
(for $|\lambda| \leq \frac 12$).
From considerations of the graph of the (continuous) function $t\log(t)$ for $t \in [0,1/2]$ we deduce that
\begin{align*}
|\lambda\log(|\lambda|)| \leq  \exp(-1) \leq \frac12,
\end{align*}
 (for  $| \lambda| \leq \frac12$) and that
\begin{align*}
|\log(\lambda)| \leq \log(|\lambda|^{-1}) + \pi/2 .
\end{align*}
And from (\ref{z0}) we have
\begin{align}
|z^{(0)}(\xi)| \leq \frac12\log(|\lambda|^{-1}) + 1.
\end{align}
Using the triangle inequality for (\ref{R}) and plugging in all of the above inequalities  we find that
\begin{align*}
|R| \leq 12\log|\lambda|^{-1} + 30.
\end{align*}
%estmating the involved constants crudely we find that
%\begin{align*}
%R \leq
%\end{align*}

Using the cancellation in $M$  together with (\ref{z0}) we find that
\begin{align*}
|M| \leq \frac{\pi}2\log|\lambda|^{-1} + \pi
\end{align*}
and finally
\begin{align}
|B_1| \leq 14|\log|\lambda|^{-1}| +33 \label{estimateB2}
\end{align}
for $|\lambda| \leq \frac12$ and $ |\lambda|\leq |\xi|/2$.\\

Now we assume that $|\lambda|\geq \frac12$. Then we can use the bounds in Lemma \ref{estimatesintegrals} to directly deduce that
\begin{align*}
|B_1| \leq 36.
\end{align*}

Finally we assume that $|\xi| < 2|\lambda|$ with no other restrictions on $(\lambda, \xi)$ so we are in the case where $\xi$ is near to 0  . We replace $z$ by $\hat{z} =z -\frac12\omega_2$ which, as can be seen from the definition, does not change $B_1$.
The integral we consider is now
$\hat{z} = -\int_{0}^{\xi}\frac{dX}{\sqrt{X(X-1)(X-\lambda)}}$. We will estimate $|\hat{z}|$
 independently of $\xi$ and $\lambda$.\\

We set $X = -t$ and compute
\begin{align*}
|\hat{z}| =  | \int_{0}^{|\xi|}\frac{dt}{\sqrt{t(t +1)(t +\lambda))}}| \leq   \int_{0}^{|\xi|}\frac{dt}{\sqrt{t(t +1)(|t +\lambda|)}}.
\end{align*}
%As $\Re(\lambda) \geq 0, t \geq 0$ the standard inequalities yield
%\begin{align}
%|t +\lambda| \geq \frac12(t +|\lambda|). \label{ineq}
%\end{align}
Clearly $t + 1\geq 1$ and with (\ref{ineq}) we may estimate as follows
\begin{align*}
|\hat{z}| \leq 2^{\frac12}\int_0^{|\xi|}\frac{dt}{\sqrt{t(t + |\lambda|)}}.
\end{align*}
Note that this, again, is an honest real integral. \\

So finally, using (\ref{primitive})  and  the fact that the integrand  is positive and $|\xi|<2|\lambda|$ we get
\begin{align*}
|\hat{z}| \leq 2^{\frac12}\int_{0}^{2|\lambda|}\frac{dt}{\sqrt{t(t + |\lambda|)}} =2^{\frac32}\log\left(\sqrt{3} + \sqrt{2}\right).
\end{align*}

%Now we give an estimate for $\omega_2$ which can be deduced from its singular expansion.
%\begin{align*}
%|\omega_2| \leq  |\log(|\lambda|)| + \pi/2 + 6, |\lambda| \leq \frac12.
%\end{align*}
%This leads to the bound
%\begin{align*}
%|B_1| \leq 2\sqrt{2}\log(\sqrt{3} + \sqrt{2})(|\log(|\lambda|)| + \pi/2 + 6 ).
%\end{align*}

This, with Lemma \ref{estimatesintegrals} leads (very crudely) to
\begin{align*}
|B_1| \leq 7\log|\lambda|^{-1} + 35.
\end{align*}
for $|\xi| < 2|\lambda|$ and concludes the proof.

\end{proof}
Now we treat the line $L_\lambda$. For this we define $z$ on $\mathcal{F}\times_\lambda L_\lambda$ by
\begin{align}
z(\lambda,\xi) = \int^{-\infty}_{\xi}\frac{dX}{2\sqrt{X(X-1)(X-\lambda)}}\label{Ll}
\end{align}
where we take the integral along $(-\infty, 0)\cup L_\lambda$.
\begin{lem}\label{L2} If $(\lambda,\xi) $ is in $ \mathcal{F}\times_\lambda L_\lambda$ then
\begin{align*}
\max\{|B_1|, |B_2|\} \leq 13\log(|\lambda|^{-1}) + 65.
\end{align*}
\end{lem}
\begin{proof} As in Lemma \ref{L1} we replace $z$ by $\hat{z} = z - \int_{\lambda}^{-\infty}\frac{dX}{2\sqrt{X(X-1)(X-\lambda)}}$. From our discussion of analytic continuations, in particular (\ref{ellintegral}) and (\ref{ellintegral2}), we see that $\hat{z} = z -\frac12(\omega_1 + \omega_2)$ which replaces  $B_1,B_2$ by $B_1 \pm \frac12A, B_2 \pm\frac12A$. But first note that
\begin{align*}
\hat{z} = \int_{\xi}^{\lambda}\frac{dX}{2\sqrt{X(X-1)(X-\lambda)}}
\end{align*}
where we take the integral along $L_\lambda$.  We set $X = \xi + t(\lambda -\xi)$ and using $|X -1| \geq \frac12, |X - \lambda| \geq |\xi - \lambda|(1-t)$ and $|X| \geq t|\lambda -\xi|$ get
\begin{align*}
|\hat{z}| \leq \int_{0}^{1}\frac{dt}{\sqrt{t(1-t)}} =\pi.
\end{align*}
Now we can read off  (\ref{B1}), (\ref{B2}) that $|B_1| \leq 2\pi|\omega_2| + |\omega_1\omega_2|$ and  $|B_2| \leq 2\pi|\omega_1| + |\omega_1\omega_2|$. If we plug in the estimates in Lemma \ref{estimatesintegrals} we deduce the present Lemma.
\end{proof}

In the final lemma of this section we treat the interval $[1,\infty)$.  We define $z$ on $\mathcal{F}\times [1,\infty)$ by
\begin{align}
z(\lambda, \xi) = \int_{\xi}^{\infty}\frac{dX}{2\sqrt{X(X-1)(X-\lambda)}} \label{1infty}
\end{align}
were we take the integral along the real line.
\begin{lem}\label{L3} Let $z$ be defined as above for $(\lambda,\xi)$ in $\mathcal{F}\times [1,\infty)$. Then
\begin{align*}
\max\{|B_1|, |B_2|\} \leq  5\log(|\lambda|^{-1}) +  25.
\end{align*}
\end{lem}
\begin{proof}
We can prove using the same computations for $\omega_1$ at the end of Lemma 2 that
\begin{align*}
|z| \leq 5/2.
\end{align*}
Then using the estimates in Lemma 2 the present Lemma follows.
\end{proof}

\section{Proof of Proposition \ref{propbetti}}
We recall that $z$ is defined as the continuation of (\ref{inverse}) to the north. We also note that the continuation of $z$ to $L_\lambda$ from the north is  given by (\ref{Ll}) where we can take the integral along $(-\infty,0)\cup L_\lambda$.
Using a homotopy argument we see that the continuation of $z$ as a function of $\xi$  to $[1,\infty)$ from the north  is equal to
(\ref{1infty}),
where we again take the integral along the real line. However we do not keep track of the sign of the square-root anymore. If we continue $z$ to $[1,\infty)$ from the south and call this continuation $z_S$ while setting $z_N$ for the continuation from the north, then $z_N + z_S$ is an integral of $\frac{dX}{\sqrt{X(X-1)(X-\lambda)}}$ from $1$ to $\infty$. Thus we get $z_N + z_S = \pm \omega_1$ (where the sign $\pm$ depends on the square-root) and so we can write $z_S$ as
\begin{align}\label{1inftyS}
z_S = -\int_{\xi}^{\infty}\frac{dX}{2\sqrt{X(X-1)(X-\lambda)}} \pm\omega_1.
\end{align}
We can then continue $z$ to $(-\infty,0)\cup L_\lambda$ from the south using  (\ref{1inftyS}) and another homotopy argument show that the continuation is of the form
 \begin{align}
z = -\int_{\xi}^{-\infty}\frac{dX}{2\sqrt{X(X-1)(X-\lambda)}} \pm\omega_1 \label{LS}
\end{align}
where we take the integral along $(-\infty,0)\cup L_\lambda$. In fact, though we won't need this, by continuing $z$  as given by (\ref{Ll}) along a small loop around $\lambda$ and then taking the limit as $\xi \rightarrow \lambda$ we can check that $\pm\omega_1 = \omega_1$ and so the square root in (\ref{1inftyS}) is the same as in (\ref{ellintegral}) for $\omega_1$.

We may continue $b_1, b_2$ using (\ref{bettidefinition1}), (\ref{bettidefinition2}) in the same fashion as real analytic functions from $\mathcal{F}\times_\lambda X_\lambda$ to $\mathcal{F}\times_\lambda ((\infty,0]\cup L_\lambda\cup[1,\infty))$. With the various Lemmas proven in the previous section we can estimate this continuation explicitly.

\begin{lem}\label{finalestimate} The continuations of $b_1,b_2$ from $\mathcal{F}\times_\lambda X_\lambda$ to $\mathcal{F}\times_\lambda ((-\infty,0]\cup L_\lambda\cup[1,\infty))$ satisfy
\begin{align*}
\max\{|b_1|, |b_2|\} \leq 41.
\end{align*}
\end{lem}
\begin{proof} By Lemma \ref{L1}, \ref{L2}, \ref{L3}  we see that for any fixed $\lambda$ in $\mathcal{F}$ the continuation of $B_1,B_2$ as real analytic functions on $X_\lambda$ to $(-\infty,0]\cup L_\lambda\cup[1,\infty)$ from the north is bounded by
\begin{align*}
\max\{|B_1|,|B_2|\} \leq 14\log|\lambda|^{-1} + 65.
\end{align*}
Now first assume that $|\lambda| \geq \exp(-21/4)$. From Lemma \ref{lemarea} we have $|A| \geq 2\sqrt{3}$ and we can compute  that $\max\{|b_1|, |b_2|\}\leq 40$. Now assume that $|\lambda| \leq \exp(-21/4)$. Then  we can deduce from Lemma \ref{lemarea} that $|A| \geq (69/100)\log|\lambda|^{-1}$ and from the above inequality follows that $\max\{|B_1|, |B_2|\} \leq 27\log|\lambda|^{-1}$. We deduce that $\max\{|b_1|, |b_2|\} \leq 40$ for all $\lambda$. Thus $\max\{|b_1|, |b_2|\}\leq 40$ for the continuation of $b_1,b_2$ to $\mathcal{F}\times_\lambda((-\infty,0]\cup L_\lambda\cup [1,\infty))$ from the north.
 By the discussion just before the statement of this Lemma, in particular (\ref{1inftyS}), (\ref{LS}), the continuation of $b_1,b_2$ to $(-\infty,0]\cup L_\lambda\cup[1,\infty)$ from the south differs from the continuation from the north by at most 1 in absolute value and we conclude the proof.
\end{proof}

Equipped with Lemma \ref{finalestimate} we can prove Proposition \ref{propbetti} using a topological argument.

\begin{proof} \textit{(of Proposition \ref{propbetti}) } We fix $\lambda$ in $\mathcal{F}$ and let $I$ be the image of the two curves $f_1,f_2 : (0,1) \rightarrow \C$
\begin{align*}
f_1(r) = \wp_\lambda(r\omega_1) + \frac13(\lambda +1) , ~~ f_2(r) = \wp_\lambda(r\omega_2) + \frac13(\lambda +1).
\end{align*}

The set $X_\lambda \setminus I$ might not be connected. But as $f_1,f_2$ are non-intersecting curves the boundary of each connected component of $X_\lambda \setminus I$ contains  a point of $(\infty,0]\cup L_\lambda \cup[1,\infty) $. So for each $\xi \in X_\lambda \setminus I$ we can continue $(z, b_1,b_2)$ from $\xi$  to $(\infty,0]\cup L_\lambda \cup[1,\infty) $ along a path lying entirely in $X_\lambda \setminus I$. Now if $z(\xi, \lambda)$ lies in a fundamental domain for $\C$ modulo $\Z\omega_1 + \Z\omega_2$ then the continuation of $z$ from $\xi$ to $(-\infty,0]\cup L_\lambda \cup[1,\infty)$ along such a path will lie in the closure of the same fundamental domain. Thus the Betti-coordinates of the continuation of $z$  and $z(\xi)$ differ by at most 1 in absolute value. This concludes the proof.
\end{proof}

\section{ Bounding $\Im(\mathcal{L})$ for large $\xi$}
We now begin working towards Proposition \ref{propsigma}. First, by \eqref{eta1} we have $\eta_1/\omega_1 = 1/3 +\lambda( -2/3 +2(1-\lambda)\omega_1'/\omega_1)$. Then on writing $X - (\lambda+1)/3 =  -1/3 +(X - \lambda/3)$ we can expand $\varphi_\lambda'/\varphi_\lambda$ as

%We note that from (\ref{eta1}) follows that $\eta_1/\omega_1 = 1/3 +\lambda( -2/3 +2(1-\lambda)\omega_1'/\omega_1)$ and decompose $X - (\lambda+1)/3 =  -1/3 +(X - \lambda/3)$. This leads to the following expansion of $\varphi_\lambda'/\varphi_\lambda$
\begin{align*}
\varphi_\lambda'/\varphi_\lambda(z(\lambda,\hat{X})) = \int_{1}^{\hat{X}}\frac{(X-\lambda/3)dX}{2\sqrt{X(X-1)(X-\lambda)}} + \lambda R_{\varphi}  +\pi i/\omega_1,
\end{align*}
where
\begin{align*}
R_{\varphi} = ( -2/3 +2(1-\lambda)\omega_1'/\omega_1)\int_1^{\hat{X}}\frac{dX}{2\sqrt{X(X-1)(X-\lambda)}}.
\end{align*}
%It holds that
%\begin{align*}
%\log(\varphi(z(\xi)))-\log(\varphi(\omega_1/2)) = -\int_1^{\xi}(\varphi'/\varphi)\frac{d\hat{X}}{ \sqrt{\hat{X}(\hat{X}-1)(\hat{X}-\lambda)}}
%\end{align*}
Using (\ref{phi}) we then have
\begin{align*}
\mathcal{L} = -\int_1^{\xi}\left(\int_{1}^{\hat{X}}\frac{(X-\lambda/3)dX}{2\sqrt{X(X-1)(X-\lambda)}}\right)\frac{d\hat{X}}{2\sqrt{\hat{X}(\hat{X}-1)(\hat{X}-\lambda)}} -\\ \lambda\int_1^{\xi}\frac{R_{\varphi}d\hat{X}}{2\sqrt{\hat{X}(\hat{X}-1)(\hat{X}-\lambda)}}  +z\pi i/\omega_1 -\pi i.
\end{align*}

Now we note that
\begin{align*}
& \left(\int_{1}^{\hat{X}}\frac{(X-\lambda/3)dX}{2\sqrt{X(X-1)(X-\lambda)}}\right)\frac{1}{2\sqrt{\hat{X}(\hat{X}-1)(\hat{X}-\lambda)}}
=\\ & \left(\int_{1}^{\hat{X}}\frac{(X-\lambda/3 -\sqrt{X(X-\lambda)})dX}{2\sqrt{X(X-1)(X-\lambda)}}\right)\frac{1}{2\sqrt{\hat{X}(\hat{X}-1)(\hat{X}-\lambda)}} + \frac1{2\sqrt{\hat{X}(\hat{X} -\lambda)}}.
\end{align*}
With this we can write $\mathcal{L}$ as
\begin{align}
\mathcal{L} = -\int_1^{\xi}\frac{d\hat{X}}{2\sqrt{\hat{X}(\hat{X}-\lambda)}} -  \mathcal{R} \label{LL1}
-\mathcal{R}_\varphi  +z\pi i/\omega_1 -\pi i.
\end{align}
where
\begin{align*}
\mathcal{R} & = \int_1^{\xi}\left(\int_{1}^{\hat{X}}\frac{(X-\lambda/3 -\sqrt{X(X-\lambda)})dX}{2\sqrt{X(X-1)(X-\lambda)}}\right)\frac{d\hat{X}}{2\sqrt{\hat{X}(\hat{X}-1)(\hat{X}-\lambda)}}\\
\mathcal{R}_\varphi & = \lambda\int_1^{\xi}\frac{R_{\varphi}d\hat{X}}{2\sqrt{\hat{X}(\hat{X}-1)(\hat{X}-\lambda)}}.
\end{align*}

We begin by bounding  $\mathcal{R}$ and $\mathcal{R}_\varphi$. As with the Betti-coordinates we start with some rather elementary Lemmas.
\begin{lem}\label{denom}  For $(\lambda,X) \in \mathcal{F}\times_\lambda X_\lambda $
\begin{align*}
|X - \lambda/3 -\sqrt{X(X-\lambda)}| \leq 5|\lambda|.
\end{align*}

\end{lem}
\begin{proof} For $|X|\leq 2|\lambda|$, the bound follows from the triangle inequality. In the case that $|\lambda/X|\leq \frac12$, we have $|\frac{\frac12(\frac12)_{n}}{n!}|\leq \frac12$. The bound then follows after considering the Taylor expansion of $\sqrt{X(X-\lambda)}$ at $\lambda=0$, which gives
\begin{align*}
X - \lambda/3 -\sqrt{X(X-\lambda)} = \lambda/6 +\lambda\sum_{n = 1}^\infty\frac{\frac12(\frac12)_{n}}{(n+1)!}\frac{\lambda^{n}}{X^{n}}.
\end{align*}
\end{proof}

%\begin{lem} For the integral along the real line and $|\xi| \geq 1$ the following inequality is valid
%\begin{align*}
%\left|\int_1^{|\xi|}\left(\int_{1}^{|\hat{X}|}\frac{(X-\lambda/3 -\sqrt{X(X-\lambda)})dX}{2\sqrt{X(X-1)(X-\lambda)}}\right)\frac{d\hat{X}}{2\sqrt{\hat{X}(\hat{X}-1)(\hat{X}-\lambda)}}\right|\leq 40
%\end{align*}
%\end{lem}

\begin{lem} \label{lemarc}  We have
\begin{align*}
\int_{|\xi|}^{\xi}\left|\frac{d\hat{X}}{2\sqrt{\hat{X}(\hat{X}-1)(\hat{X}- \lambda)}}\right| \leq \pi/\sqrt{|\lambda|}
\end{align*}
for $|\lambda| \leq \frac12|\xi|$, where we take the integral along the arc of a circle with radius $|\xi|$.
\end{lem}
\begin{proof} We first note that as $|\xi-\lambda|\geq |\lambda|$ it is sufficient to prove that
\begin{align}
\int_{|\xi|}^{\xi}\left|\frac{d\hat{X}}{2\sqrt{\hat{X}(\hat{X}-1)}}\right|\leq \pi.\label{arc}
\end{align}
We set $\hat{X}=|\xi|\exp(2\pi i\theta)$ and obtain
\begin{align*}
\pi\int_{0}^{\theta_0}|\frac{d\theta|\sqrt{\xi}|}{\sqrt{||\xi|\exp(2\pi i\theta)-1}|}|\leq2 \pi\int_{0}^{\frac14}\frac{d\theta}{\sqrt{|\sin(2\pi i\theta)}|}\leq 4\pi\int_0^{\frac14}\frac{d\theta}{\sqrt{\theta}}\leq\pi.
\end{align*}
\end{proof}
%Before we estimate $\mathcal{R}$ we need yet another Lemma.

\begin{lem}\label{denom2} For $\lambda \in \mathcal{F}$ the following holds
\begin{align*}
 | -2/3 +2(1-\lambda)\omega_1'/\omega_1|\leq 11.
\end{align*}
\end{lem}
\begin{proof}  Lemma \ref{area} gives $|\omega_1|\geq 1$, thus it suffices to find an estimate for $|\omega_1'|$. We write
\begin{align*}
\omega_1' =\int_1^\infty\frac{dX}{2(X-\lambda)\sqrt{X(X-1)(X-\lambda)}}
\end{align*}
and note that $|X-\lambda| \geq \frac12$. It then follows from the estimates at the end of Lemma \ref{estimatesintegrals} that $|\omega_1'|\leq 5$, which is enough to complete the proof.
\end{proof}
\begin{lem} \label{lemgreater}For $|\xi|\geq 1, |\lambda/\xi| \leq \frac12$ it holds that
\begin{align*}
\max\{|\mathcal{R}(\xi)|, |\mathcal{R}_\varphi(\xi)|\} \leq 132
\end{align*}
\end{lem}
\begin{proof} We decompose the double integrals involved in the definition of $\mathcal{R}$ and $\mathcal{R}_\varphi$ (formally) as follows
\begin{align*}
\int_1^{\xi}\int_1^{\hat{X}} = \int_1^{|\hat{\xi}|}\int_1^{|\hat{X}|} + \int_{|\xi|}^{\xi}\int_{1}^{|\hat{X}|} + \int_{|\xi|}^{\xi}\int_{|\hat{X}|}^{\hat{X}} = I_1 + I_2 +I_3
\end{align*}
where we take the integral along the real line and then along the arc of a circle. We first treat $I_1$. By Lemma \ref{denom}  and \ref{denom2} it is enough to bound $|\int_{1}^{|\hat{X}|}\frac{dX}{2\sqrt{X(X-1)(X-\lambda)}}|$ by an absolute constant. For this we can use Lemma \ref{estimatesintegrals} which gives
\begin{align}
\int_{1}^{|\hat{X}|}\left|\frac{dX}{2\sqrt{X(X-1)(X-\lambda)}}\right|\leq \int_{1}^{\infty}\left|\frac{dX}{2\sqrt{X(X-1)(X-\lambda)}}\right| =|\omega_1|/2\leq 5/2,\label{pos}
\end{align}
from which we deduce that $|I_1| \leq 74$. \\

For $I_2$ and $I_3$ we make a distinction between $|\lambda|\geq \frac12$ and $|\lambda|\leq \frac12$. For $|\lambda|\geq\frac12$ we may deduce directly from Lemma \ref{denom}, \ref{denom2} and (\ref{pos}) that $\max\{|I_2|,|I_3|\} \leq 66$. For $|\lambda|\leq \frac12$ we can use   Lemma \ref{denom} and Lemma \ref{denom2} as well as Lemma \ref{lemarc} (note the cancellation of $|\lambda|$) to estimate $\max\{|I_2|,|I_3|\}\leq 132$.
\end{proof}
Now we handle the case $|\xi|\leq 1$.

\begin{lem}\label{logl} For $1\geq|\hat{X}|\geq 2|\lambda|$ we have
\begin{align*}
\int_1^{\hat{X}}\frac{|dX|}{|2\sqrt{X(X-1)(X-\lambda)|}}\leq \log|\lambda|^{-1}+9,
\end{align*}
where we integrate first along the real line and then along a circle with radius $|\hat{X}|$.
\end{lem}
\begin{proof} We note that $|\sqrt{X(X-1)(X-\lambda)}| =|X\sqrt{(X-1)(1-\lambda/X)}|\geq \sqrt{2}|X||\sqrt{X-1}|$. We first treat the integral along the real line. If $|\hat{X}|\geq \frac34$ then
\begin{align*}
\int_1^{|\hat{X}|}\frac{|dX|}{|2\sqrt{X(X-1)(X-\lambda)|}}\leq \int_1^{\frac34}\frac{|dX|}{2\sqrt{2}|X|\sqrt{X-1}} \leq \int_1^{\frac34}\frac{|dX|}{|\sqrt{X-1}|} \leq 1.
\end{align*}
If $|\hat{X}|\leq \frac34$ then $|X-1|\geq \frac14$ and we decompose the integral into two parts, the second being
\begin{align*}
\int_{\frac34}^{|\hat{X}|}\frac{|dX|}{|2\sqrt{X(X-1)(X-\lambda)|}}  \leq \frac1{\sqrt{2}}\int_{\frac34}^{|\hat{X}|}\frac{dX}{X}\leq 1/\sqrt{2}\log|\lambda|^{-1} +|(\log(3/2)/\sqrt{2}| \leq \log|\lambda|^{-1} +1
\end{align*}
(where we used that $|\hat{X}|\geq 2|\lambda|$). This proves that
\begin{align*}
\int_{1}^{|\hat{X}|}\frac{|dX|}{|2\sqrt{X(X-1)(X-\lambda)|}}  \leq \log|\lambda|^{-1} +2.
\end{align*}
For the integral along the arc of the circle we again first assume that $|\hat{X}|\geq \frac34$. Then
\begin{align*}
\int_{|\hat{X}|}^{\hat{X}}\frac{|dX|}{|2\sqrt{X(X-1)(X-\lambda)|}} \leq \int_{|\hat{X}|}^{\hat{X}}\frac{dX}{|\sqrt{X-1}|}\leq  \int_{|\hat{X}|}^{\hat{X}}\frac{dX}{|\sqrt{X(X-1)}|}\leq 2\pi
\end{align*}
where in the last inequality we have used (\ref{pos}). Now for $|\hat{X}|\leq \frac34$. Then $|X-1|\geq \frac14$ and we obtain
\begin{align*}
\int_{|\hat{X}|}^{\hat{X}}\frac{|dX|}{|2\sqrt{X(X-1)(X-\lambda)|}} \leq\int_{|\hat{X}|}^{\hat{X}}\frac {|dX|}{|X|} \leq \pi.
\end{align*}
\end{proof}

\begin{lem}\label{lemless}
For $|\xi| \leq 1, |\lambda/\xi| \leq \frac12$ we have
\begin{align*}
\max\{|\mathcal{R}(\xi)|, |\mathcal{R}_\varphi(\xi)|\} \leq 1100.
\end{align*}
\end{lem}
\begin{proof} By Lemma \ref{denom}, \ref{denom2} and \ref{logl}  it suffices to bound
$11|\lambda||\log|\lambda|^{-1}+6|^2$. Using $\log|\lambda|^{-1}\leq |\lambda|^{-\frac14}$ we have
\begin{align*}
11|\lambda||\log|\lambda|^{-1}+9|^2 \leq11||\lambda|^{\frac14} +9|\lambda|^{\frac12}|^2 \leq 1100.
\end{align*}
\end{proof}

Now we handle the term $\int_{1}^\xi\frac{dX}{2\sqrt{X(X-\lambda)}}$.

\begin{lem}\label{lemlog1}  The following holds
\begin{align*}
\Im\left(\int_{1}^\xi\frac{dX}{2\sqrt{X(X-\lambda)}}\right) \leq 7,
\end{align*}
for $|\lambda/\xi| \leq \frac12$.
\end{lem}
\begin{proof}
As usual, we first integrate along the real line and then along the arc of a circle. For the integral along the arc of a circle we have
\begin{align*}
\left|\int_{|\xi|}^{\xi}\frac{dX}{2\sqrt{X(X-\lambda)}}\right|\leq \int_{|\xi|}^{\xi}\left|\frac{dX}{2X\sqrt{1-\lambda/X}}\right| \leq  \int_{|\xi|}^{\xi}\frac{dX}{|X|} \leq \pi.
\end{align*}
For the integral along the real line we first assume that $1\leq |\xi|\leq 2$ then as $|X-\lambda| \geq \frac12$ and $|X|\geq 1$ we get
\begin{align*}
\left|\int_1^{|\xi|}\frac{dX}{2\sqrt{X(X-\lambda)}}\right|\leq \int_{1}^{2}\left|\frac{dX}{2\sqrt{X(X-\lambda)}}\right| \leq  1.
\end{align*}
In order to treat the cases $|\xi| \geq 2$ and $|\xi|\leq 1$ we develop $\int_{\epsilon}^\xi dX/(2\sqrt{X(X-\lambda)})$ into a Taylor series at $\lambda =0$ where we set $\epsilon$  equal to 1 if $|\xi| \leq 1$ and equal to 2 if $|\xi| \geq 2$
\begin{align*}
\int_{\epsilon}^\xi\frac{dX}{2\sqrt{X(X-\lambda)}} =\int_\epsilon^{\xi}\frac{dX}{2X}+\frac12\sum_{n=1}^{\infty}\frac{(\frac12)_n}{n!}\int_\epsilon^{\xi}\left(\frac{\lambda^ndX}{X^{n+1}}\right)
\end{align*}
and the infinite series converges whenever $|\lambda|/\epsilon\leq\frac12$ and $|\lambda|/|\xi|\leq \frac12$. Now if $\epsilon =2$ these conditions are always satisfied and if $|\xi|\leq 1$ then $|\lambda| \leq \frac12$ and these conditions are satisfied as well. We get the following for $n\geq 1$
\begin{align*}
\left|\frac12\frac{(\frac12)_n}{n!}\int_\epsilon^{\xi}\left(\frac{\lambda^n}{X^{n+1}}\right)\right| \leq \left(\frac12\right)^n
\end{align*}
and so
\begin{align*}
\left|\int_{\epsilon}^\xi\frac{dX}{2X} - \int_\epsilon^{\xi}\frac{dX}{2\sqrt{X(X-\lambda)}}\right| \leq 1.
\end{align*}
For $|\xi|\leq 1$ or $|\xi| \geq 2$ we get
\begin{align*}
\int_1^{\xi}\frac{dX}{2\sqrt{X(X-\lambda)}} = \frac12\log(\xi) + R_l
\end{align*}
where $|R_l|\leq 2 +\log(2) \leq 3$ by first integrating from 1 to $\epsilon$ and then from $\epsilon$ to $\xi$. With the previous estimates for $1\leq |\xi|\leq 2$ we obtain
\begin{align*}
\left|\Im\left(\int_1^{\xi}\frac{dX}{2\sqrt{X(X-\lambda)}}\right)\right| \leq \pi +3\leq 7.
\end{align*}
\end{proof}
\section{Bounding $\Im (\mathcal{L})$ for small $\xi$}

In this section we assume that $|\xi| \leq 2|\lambda|$.

It is convenient to define $z(\lambda,\hat{X}), \zeta_\lambda( z(\lambda,\hat{X}))$ with another integral.  We have
\begin{align*}
\zeta_\lambda(z(\lambda, \hat{X})) = \int_0^{\hat{X}}\frac{dX(X-\frac13(\lambda +1))}{2\sqrt{X(X-1)(X-\lambda)}} + \eta_2/2.
\end{align*}

Similarly, we have
\begin{align*}
z(\lambda, \hat{X}) = -\int_0^{\hat{X}}\frac{dX}{2\sqrt{X(X-1)(X-\lambda)}} +\omega_2/2.
\end{align*}

Now let
\begin{align*}
\LLL = \int_0^{\xi}\left(\int_0^{\hat{X}}\frac{dX(X-\frac13(\lambda+1))}{2\sqrt{X(X-1)(X-\lambda)}}\right)\frac{d\hat{X}}{2\sqrt{\hat{X}(\hat{X}-1)(\hat{X}-\lambda)}} - \\
(\eta_1/\omega_1)\int_0^{\xi}\left(\int_0^{\hat{X}}\frac{dX}{2\sqrt{X(X-1)(X-\lambda)}}\right)\frac{d\hat{X}}{2\sqrt{\hat{X}(\hat{X}-1)(\hat{X}-\lambda)}}.
\end{align*}
Then using the Legendre relation
\begin{align*}
\omega_2\eta_1-\omega_1\eta_2=2\pi i,
\end{align*}
with \eqref{diffphi} and the new integrals, we find that
\begin{align}
\LLL = \log(\varphi_\lambda(z(\xi)))-\log(\varphi_\lambda(\omega_2/2)). \label{LLL}
\end{align}
Here we define $\log$ locally at $\varphi_\lambda(\omega_2/2)\neq 0$ and then continue $\log(\varphi_\lambda(z(\xi)))$ to the domain $\{|\xi|\leq 2|\lambda|\}\cap X_\lambda$ . We note that on the circle $|\xi| =2|\lambda|$ we have
\begin{align*}
\mathcal{L} -\LLL = \log(\varphi_\lambda(\omega_2/2))-\log(\varphi_\lambda(\omega_1/2))
\end{align*}
and then by analytic continuation this holds also on the whole domain where $\mathcal{L}$ is defined. We also note that
\begin{align*}
|\Im\left(\log(\varphi_\lambda(\omega_2/2))-\log(\varphi_\lambda(\omega_1/2))\right)|\leq 2\pi.
\end{align*}
Now we need another Lemma
\begin{lem}\label{hatX} If $|\hat{X}|\leq 2|\lambda|$ then
\begin{align}
\int_{0}^{\hat{X}}\left|\frac{dX}{2\sqrt{X(X-1)(X-\lambda)}}\right| \leq 12. \label{8}
\end{align}
\end{lem}
\begin{proof} We first consider the case that $|\hat{X}-1|\geq \frac14$ and show that
\begin{align*}
\int_0^{\hat{X}}\left|\frac{dX}{\sqrt{X(X-\lambda)}}\right| \leq6,
\end{align*}
where we take the integral along a straight line joining 0 and $\hat{X}$. We set $X=\hat{X}t$ and, to begin with, assume that $|\lambda/\hat{X}|\leq 1$
\begin{align*}
\int_0^{\hat{X}}\left|\frac{dX}{\sqrt{X(X-\lambda)}}\right| \leq \int_0^{1}\frac{dt}{\sqrt{|t(t-\lambda/\hat{X})}|}\leq \int_0^{|\lambda|/|\hat{X}|}\frac{dt}{\sqrt{t(|\lambda/\hat{X}| -t)}} +\int_{|\lambda|/|\hat{X}|}^1\frac{dt}{\sqrt{t(t-|\lambda/\hat{X}| )}}.
\end{align*}
We split the first integral on the right into two parts. The first can be estimated as follows
\begin{align*}
\int_0^{|\lambda|/|2\hat{X}|}\frac{dt}{\sqrt{t(|\lambda/\hat{X}| -t)}} \leq \sqrt{2|\hat{X}|/|\lambda|}\int_0^{|\lambda|/|2\hat{X}|}\frac{dt}{\sqrt{t}} \leq 2.
\end{align*}
and the second  as follows
\begin{align*}
\int_{|\lambda|/|2\hat{X}|}^{|\lambda|/|\hat{X}|}\frac{dt}{\sqrt{t(|\lambda/\hat{X}| -t)}} \leq \sqrt{2|\hat{X}|/|\lambda|}|\int_{|\lambda|/|2\hat{X}|}^{|\lambda|/|\hat{X}|}\frac{dt}{\sqrt{|\lambda/\hat{X}| -t}} \leq 2.
\end{align*}
For the second integral on the right hand side above, we use the fact that  $\int \frac{dt}{\sqrt{t(t-|\lambda/\hat{X}| )}} = 2\log(\sqrt{t-|\lambda/\hat{X}|} +\sqrt{t})$ and obtain
\begin{align*}
\int_{|\lambda|/|\hat{X}|}^1\frac{dt}{\sqrt{t(t-|\lambda/\hat{X}| )}} \leq \log(4)\leq 2,
\end{align*}
where we used that $|\lambda/\hat{X}|\geq \frac12$.
To complete the first case, note that if $|\lambda/\hat{X}| \geq 1$ then
\begin{align*}
\int_0^{\hat{X}}\left|\frac{dX}{\sqrt{X(X-\lambda)}}\right| \leq \int_0^1\frac{dt}{\sqrt{t(1-t)}} = \pi.
\end{align*}

Now assume that $|\hat{X}-1| \leq \frac14$. Then $|X-\lambda|\geq \frac14$ and $|X|\geq \frac34$ and we end up with
\begin{align*}
\int_0^{\hat{X}}\left|\frac{dX}{2\sqrt{X(X-1)(X-\lambda)}}\right| \leq \int_0^{\hat{X}}\left|\frac{dX}{\sqrt{X-1}}\right|.
\end{align*}
Here we decompose this last integral into an integral along the real line which is  bounded by $1/2$ and an integral along the arc of a circle which by (\ref{arc}) is bounded by $2\pi$.
%\begin{align*}
%\int_0^{2|\lambda|}\left|\frac{dX}{\sqrt{X(X-\lambda)}}\right| \leq \int_0^{|\lambda|}|\frac{dX}{\sqrt{X(|\lambda| -X)}} + \int_{|\lambda|}^{2|\lambda|}|\frac{dX}{\sqrt{X(X-|\lambda|)}}
%\end{align*}
\end{proof}

We can now bound $|\LLL|$ and then $\Im(\mathcal{L})$. We have seen at the end of the proof of Lemma \ref{denom2} that $|\omega_1'|\leq 5$. From this,  we derive that $|\eta_1/\omega_1|\leq 11$.  Since $|X|\leq 2|\lambda|$, we have $|X-\frac13(\lambda +1)|\leq 3$. Using Lemma \ref{hatX} repeatedly we then find that
\begin{align*}
|\LLL|\leq 2016.
\end{align*}
This leads to a bound on $|\Im(\mathcal{L})|$ on the domain $|\xi|\leq 2|\lambda|$ when we note that
\begin{align*}
|\Im(\mathcal{L})| =|\Im(\mathcal{L}) - \Im(\tilde{\mathcal{L}}) + \Im(\LLL)|
\leq |\Im(\LLL)| +  |\Im(\phi_\lambda(\omega_2/2) - \phi_\lambda(\omega_1/2))|
\leq 2016 +2\pi \leq 2023.
\end{align*}
\section{Proof of Proposition \ref{propsigma}}
We first note that if $(\lambda,\xi) \in \mathcal{F}\times_\lambda X_\lambda$, then
\begin{align*}
\Im(i \pi z/\omega_1 -i\pi) \leq 64\pi.
\end{align*}
This follows on writing $z=b_1\omega_1+b_2\omega_2$, observing that the imaginary part of $i\omega_2/\omega_1$ is bounded in modulus by $1/2$ and applying Proposition \ref{propbetti}.

To prove Proposition \ref{propsigma}, we consider two cases. First, if $|\lambda/\xi| >\frac12$, we have shown that $|\Im(\mathcal{L})|\leq 2023$ at the end of the previous section. On the other hand, if $|\lambda/\xi|\leq \frac12$, use (\ref{LL1}), and then apply Lemma  \ref{lemlog1}, Lemma \ref{lemgreater}, Lemma \ref{lemless} and Lemma \ref{lemlog1} together with the observation above to find that
\begin{align*}
|\Im(\mathcal{L})| \leq 2409.
\end{align*}

%(STILL NEED TO UPDATE NUMBER IN PROPOSITION 2)

\section{Appendix}
Here we establish some limits on how far our results can be pushed. To begin, we show that no analogue of our results for $\wp, \zeta$ and $\varphi$ holds for $\sigma$. Here it will be convenient to use some terminology from logic.

We assume that $\lambda \in \mathcal{F}$ and recall the singular expansion of $\omega_2$ (\ref{omega2}) and the relations (\ref{eta1}),(\ref{eta2}).  We set $\omega = \omega_1 + \omega_2, \eta = \eta_1 + \eta_2$. We first note that as
\begin{align*}
\omega_2 + \omega_1 = i(\omega_1/\pi)\log(\lambda) + \omega_1+u
\end{align*}
we have
\begin{align*}
\omega_2' + \omega_1' = i(\omega_1'/\pi)\log(\lambda) +i\omega_1/(\lambda\pi)+  \omega_1'+u'.
\end{align*}
Using the expansion $\omega_1 = \pi + O(\lambda)$  we get
\begin{align*}
2\lambda(1-\lambda)(\omega_2' + \omega_1') = 2i +O(\lambda\log(\lambda))
\end{align*}
and also using $u = i4\log2 + O(\lambda)$ we have
\begin{align*}
\omega_2 + \omega_1 = i \log(\lambda) + \pi + i4\log2 + O(\lambda\log(\lambda)).
\end{align*}
With the help of these expressions we compute
\begin{align*}
\omega\eta = -\frac13\log(\lambda)^2 +\frac23(i\pi - 4\log2)\log(\lambda) - 2\log(\lambda) +O(1).
\end{align*}
If we express $\log(\lambda)^2 = \log(|\lambda|) +2i\arg(\lambda)\log(|\lambda|) -\arg(\lambda)^2$ we can further compute
\begin{align*}
\Im(\omega\eta) =  \frac23(\pi -\arg(\lambda))\log|\lambda| +O(1).
\end{align*}
%With the help of these expressions we can compute
%\begin{align*}
%\omega_1\eta_1 = \frac13(\omega_1 + \omega_2)^2 + 2i\omega_1(\omega_1 + \omega_2) + O(1).
%\end{align*}
It follows from this that as $\lambda\in \mathcal{F}$ approaches $0$, the imaginary part of $\omega\eta$ tends to infinity, for we have $|\arg (\lambda)|\leq \pi/2$ for such $\lambda$. \\

 Let $\sigma_\lambda$ be the Weierstrass sigma function associated to the lattice spanned by $\omega_1,\omega_2$. This function is odd and satisfies
\begin{align*}
\sigma_\lambda(z + \omega) = -\sigma_\lambda\exp(\eta(z + \frac{\omega}2))
\end{align*}
\cite[Theorem 1, page 241]{Lang}. Using these two properties we deduce that
\begin{align*}
\frac{\sigma_\lambda((\frac12+r)\omega)}{\sigma_\lambda((\frac12-r)\omega)} = \exp(r\eta\omega).
\end{align*}

For $r \in [0,\frac12)$ both $(\frac12-r)\omega$ and $(\frac12+r)\omega$ lie in the fundamental parallelogram spanned by $\omega_1,\omega_2$. Thus the function
\begin{align*}
\psi_\lambda(r) = \Im\left(\frac{\sigma_\lambda((\frac12+r)\omega)}{\sigma_\lambda((\frac12-r)\omega)}\right)
\end{align*}
 restricted to $[0,\frac12)$ is  definable in the expansion of the real field by the restrictions of $\Re \sigma_\lambda, \Im\sigma_\lambda$ to this fundamental parallelogram. However $\psi_\lambda(r_0) = 0$ whenever $\theta(r_0) = \frac1{\pi}(\Im(r_0\eta\omega))$ is an integer.  The function $\theta$ is real and continuous on $[0,\frac12)$ and its image contains an interval of length
\begin{align*}
\left|\frac{\Im(\eta\omega)}{2\pi}\right|
\end{align*}
and so at least $|\frac{\Im(\eta\omega)}{2\pi}|-1$ distinct integers. By our observations above, this latter expression tends to $+\infty$ as $\lambda$ approaches $0$ and so the number of zeroes of $\psi_\lambda$ is unbounded as $\lambda$ tends to 0. Using this we have the following.
%\begin{prop} Let $L$ be the language of the real ordered field together with two binary function symbols $f$ and $g$. Then there is a formula $\theta(x)$ with the following property. For all positive integers $n$ there is a $\lambda$ so that if we  interpret $f$ and $g$ as the real and imaginary parts of $\sigma_\lambda$, the set defined by $\theta$ has at least $n$ connected components.
%\end{prop}
\begin{prop} Let $L$ be the language of the real ordered field together with two binary functions $f$ and $g$. Then there is a formula $\theta(x)$ with the following property. For every positive integer $n$  there exists $\epsilon >0$ such that if $\lambda \in \mathcal{F}$ with $|\lambda|<\epsilon$  then,  upon interpreting $f$ and $g$ as the real and imaginary parts of $\sigma_\lambda$, the set defined by $\theta$ has at least $n$ connected components.
\end{prop}
To prove this from the above, we simply let $\theta(x)$ be a formula which expresses $\psi_\lambda(x)=0$ in the structures mentioned in the proposition.

From the proposition, it follows immediately that there is no analogue of Peterzil and Starchenko's well-known result on the two-variable $\wp$-function \cite{PS} for $\sigma$ as a function of two variables. Similarly, it follows easily from the proposition (and Khovanskii's theorem) that if $B(\lambda)$ is a bound on the entries of the format of a representation of $\sigma|_{\frak{F}_\Omega}$ as a piecewise subpfaffian set, then $B(\lambda)$ is unbounded as $\lambda$ varies in $\mathcal{F}$. So there is no analogue for $\sigma$ of our results for $\wp,\zeta$ and $\varphi$.

Finally, we discuss the choice of the fundamental domain $\frak{F}_\Omega$. We have chosen $\omega_1$ and $\omega_2$ such that $\omega_2/\omega_1$ lies in the standard fundamental domain in the upper half plane. Surprisingly (to us at least), this choice is important. In fact, if we change the fundamental domain of $\Omega$, the format of the corresponding definition of $\wp$ might go up. To see this, let
\begin{align*}
\frak{F}_\Omega^{(a,b,c,d)}= \{r_1(a\omega_1+b\omega_2) +r_2(c\omega_1+ d\omega_2); 0\leq r_1,r_2<1, r_1^2+r_2^2\neq 0\}.
\end{align*}
for some integers $a,b,c,d$ with $ad-bc =1$. Let $B'(a,b,c,d)$ be a bound on the entries of the format of a representation of $\wp|_{\frak{F}^{(a,b,c,d)}_\Omega}$ as a piecewise subpfaffian set.
\begin{prop} The number $B'$ tends to infinity as $\max\{|a|,|b|,|c|,|d|\}$ tends to infinity.
\end{prop}
\begin{proof} Suppose not. Then we can find an infinite sequence of distinct tuples for which $B'$ is smaller than a fixed constant. Take a tuple $(a,b,c,d)$ of that sequence and pick the entry that has modulus $n =\max\{|a|,|b|,|c|,|d|\}$. Say it is $a$. By our assumption there is a representations of the curves $C=\{ \wp(r\omega_2): r\in (0,1)\}$  and $C_n=\{ \wp (r(a\omega_1+b\omega_2)) : r\in (0,1)\}$ as piecewise subpfaffian sets whose formats are bounded independently of $n$. And then there is a similar representation of $C\cap C_n$. So by Khovanskii's theorem the number of connected components of this set is again bounded independently of $n$. However, this set contains at least $(n-1)/2$  isolated points.  If one of the other entries has modulus $n$ we can make an analogous construction. Thus $n$ is bounded along that sequence. This is a contradiction.
\end{proof}

Motivated by a question that Corvaja and Zannier asked us, we now show  how the results in this paper lead to an effective bound on the Betti map of a section of $E_\lambda$ restricted to a small triangle in $\C$ with a vertex 0. \\

In order to keep the discussion brief we restrict our attention to triangles contained in $\mathcal{F}$ but this could be extended without difficulty. \\

%Recall the definition of $X_\lambda$ where we allow $\lambda$ to vary in $\C\setminus\{0,1\}$ now. Let $\Sigma$ be any sector centered at 0. We define the Betti-coordinates on $X_\lambda$ as usual for $\lambda \in \Sigma$ as before but the periods are now given by a fixed analytic continuations of $\omega_1,\omega_2$ to $\Sigma$ and $z$ is defined just as for $\lambda \in \mathcal{F}$. \\

%We first show that the Betti-coordinates are effectively bounded on $X_\lambda$. For fixed $\lambda$ and the Betti-coordinates described as above. 
Let $Q \in \C[X,Y]\neq 0$ and choose $\Delta$ an open triangle in the set $\mathcal{F}$ with 0 as one of its vertices such that there is an analytic solution $\xi$ to 
\begin{align*}
Q(\xi,\lambda)=0
\end{align*}
on $\Delta$ and each such solution satisfies $\xi(\lambda) \neq 0,1,\lambda$. 
 Now we fix such an analytic solution $\xi$. If $\xi$ is non-constant  we pick $\lambda_0 \in \Delta$ such that $\xi(\lambda)\in X_\lambda$ for $\lambda$ in a neighbourhood of $\lambda_0$ and define the Betti coordinates for $\xi(\lambda)$ as usual. If $\xi$ is constant we also define the Betti coordinates with $z$ defined as the continuation of $z$ from the north if $\xi$ lies on the boundary of $X_\lambda$.   We continue the Betti-coordinates from this neighbourhood analytically to $\Delta$. 
\begin{prop} The Betti map of $\xi$ on $\Delta$ is  bounded effectively by a constant that depends  only on the degree of $Q$.
\end{prop}
\begin{proof}

 Let $\lambda_1\in \Delta$ and $l:[0,1]\rightarrow \Delta$ be a straight line joining $\lambda_0$ and $\lambda_1$. The set 
 \begin{align*}
 \{t \in [0,1]; Q(\xi,l(t))=0 \text{ and } \xi \in (-\infty,0]\cup [1,\infty)\cup \{\text{image of } l\}\}
 \end{align*} 
has a finite number of connected components $N(\lambda_1)$ and this number is bounded solely and effectively by the degree of $Q$. \\

The map $L =(l,\xi\circ l)$ defines a path in the space $S= \Delta\times \C\setminus ((\Delta\times \{0\}) \cup (\Delta\times \{1\})\cup \{(\lambda,\lambda); \lambda \in \Delta\})$ and we can choose a path from $(\lambda_1,\xi(\lambda_1))$ to $p_0=(\lambda_0,\xi(\lambda_0))$ lying entirely in the fibred product $\Delta\times_\lambda X_\lambda$. We can compose those two paths to get a loop $\gamma \in \pi_1(p_0,S)$. \\

The fundamental group $F=\pi_1(p_0,S)$ is generated by the three loops $\gamma_1, \gamma_2$ and $ \gamma_3$ around $\Delta\times \{0\}, \Delta\times \{1\}$ and $ \{(\lambda,\lambda); \lambda \in \Delta\}$ respectively. These are chosen such that, say, the compositum with $\xi$ of the first two are small loops around  $0$ and $1$ respectively while the compositum of the third with $\xi-\lambda$ is a small loop around $0$. There is a group homomorphism $\rho:F \rightarrow S_2\ltimes \Z^2$ where the group law on $S_2\ltimes \Z^2$ is defined by $(x_1,y_1)\cdot(x_2,y_2) =(x_1x_2, x_1y_1 +y_2)$ (where we identified $S_2$ with $\{\pm1\}$). From (\ref{ellintegral}) and (\ref{ellintegral2}) we can deduce that it is given by
\begin{align*}
\rho(\gamma_1) =(-1,(0,1)), \rho(\gamma_2) = (-1,(1,0)), \rho(\gamma_3) = (-1,(1,1))
\end{align*}
and the action of $F$ on the Betti-coordinates (given by analytic continuation) can be expressed by $(a,b):(b_1,b_2) \rightarrow a(b_1,b_2) +b$. \\ 

Now by an elementary geometric argument the word-length of $\gamma$ as a word in $\gamma_1,\gamma_2, \gamma_3$ can be bounded from above by $N(\lambda_1)$ which as remarked above is bounded independently of $\lambda_1$. It follows that if $\rho(\gamma) =(x,y)$ then $y$ is bounded independently of $\lambda_1$. Since by Proposition \ref{propbetti} the Betti-coordinates on $X_\lambda$ are bounded effectively  so are the Betti coordinates of $\xi$ on $\Delta$.
\end{proof}

We note that although there is some choice involved in $\Delta$, the bound obtained is independent of the choices made. To obtain a statement about a general open triangle with vertex 0, contained in $\C\setminus\{0,1\}$ we  cut it into several simple regions. This construction can also be carried out in an effective manner.

Finally, again in connection with correspondence with Corvaja and Zannier, we discuss the definability of Betti maps (of sections of $E_\lambda$), viewed now as functions of $\lambda$. To this end we fix $U\subseteq \C\setminus \{0,1\}$, an open set, definable (by which we shall always mean definable in $\R_{\text{an,exp}}$). We suppose that $U$ is simply connected. For instance $U$ could be a sector of the unit circle. On $U$ we take some choice of period maps $\omega_1$ and $\omega_2$ (we need not make the particular choice made elsewhere in the paper, but we do number them such that the quotient below takes values in the upper half plane). We now write $\lambda$ for the usual $\lambda$-function on the upper half plane, and so we will write $t$ for the variable in $U$. The quotient $\omega_2/\omega_1$ is a branch of the inverse of $\lambda$. By a theorem of Peterzil and Starchenko \cite{PS}, $\lambda$ is definable on its usual fundamental domain and on the image of this domain under finitely many elements in $\Gamma(2)$ (the elements needed will depend on $U$ and on the choices of the periods). As the inverse of a definable function, the quotient above is definable. It follows that the derivative of this quotient is also definable (see, for example, Chapter 7 of van den Dries's book \cite{vdD}). Computing, we find that
\[
\left( \frac{\omega_2}{\omega_1}\right)'(t)= \frac{c}{t(1-t)\omega_1(t)^2},
\]
for some absolute $c \neq 0$. So $\omega_2$ is too.

To get the definability of the elliptic logarithms, we use the definability, also due to Peterzil and Starchenko \cite{PS}, of the map $\wp$, as a function of both $z$ and $\tau$, on the domain $\{ (\tau, z): \tau \in \mathcal{S} \text{ and } z\in \frak{F}_{\Omega_\tau} \}$. Here $\mathcal{S}$ is the usual fundamental domain in the upper half plane, and $\Omega_\tau$ is the lattice generated by $1$ and $\tau$. This definability clearly extends to the domain with $\mathcal{S}$ replaced by the union of fundamental domains for the $\lambda$ function that we used above.

Suppose that $\xi$ is an algebraic function of $t$, and that we have fixed a definite well-defined branch on $U$. We now write $z$ for some branch on $U$ of the elliptic logarthim determined by $\xi$. This satisfies
\[
\wp (\lambda^{-1}(t), z(t))= \xi(t) -\frac13(t+1),
\]
with the inverse of $\lambda$ that we gave above. From this follows the definability of $z$ on $U$. And once we have the periods and the logarithm we get the definability of the Betti maps on $U$, defined as usual by \eqref{bettidefinition1} and \eqref{bettidefinition2} but with the periods and logarithm as above. It then follows from general facts on definability, which can again be found for instance in Chapter 7 of van den Dries's book \cite{vdD}, that the differential of the Betti map (considered in more general setting in \cite{CMZ2}) is also definable.

\end{document}